\documentclass[a4paper,12pt]{nyjm}
\usepackage[utf8]{inputenc} 
\DeclareMathAlphabet{\mathfr}{U}{euf}{m}{n}
\usepackage{amsmath}
\usepackage{amssymb}
\usepackage{amscd}
\usepackage{amsthm}
\usepackage{array}
\usepackage[all]{xy}
\usepackage{graphicx}
\usepackage{mathtools}
\usepackage[a4paper,top=3cm,bottom=2cm,left=3cm,right=3cm,marginparwidth=1.75cm]{geometry}
\usepackage[colorlinks=true, allcolors=blue]{hyperref}
\usepackage{comment}
\usepackage[backend=biber,style=alphabetic]{biblatex}

\newtheorem{thm}{Theorem}[section]
\newtheorem{defi}[thm]{Definition}

\newtheorem{coro}[thm]{Corollary}
\newtheorem{prop}[thm]{Proposition}
\newtheorem{rmk}[thm]{Remark}

\numberwithin{equation}{section}

\newcommand{\Q}{\mathbb Q}

\newcommand{\Qp}{{\mathbb Q}_p}
\newcommand{\Zp}{{\mathbb Z}_p}

\newcommand{\Z}{\mathbb Z}

\newcommand{\Gal}{\mathrm{Gal}}

\newcommand{\OL}{{\mathcal O}_L}
\newcommand{\ROE}{{\mathcal O}_E}

\newcommand{\End}{\operatorname{End}}

\title[Module structure of dihedral degree $2p$ extensions of $\Qp$]{Hopf Galois module structure of dihedral degree $2p$ extensions of $\Qp$}
\author{Daniel Gil-Muñoz, Anna Rio}
\date{\today}

\begin{document}
\maketitle

\begin{abstract}
\medskip
Let $p$ be an odd prime. For field extensions $L/\Qp$ with Galois group isomorphic to the dihedral group $D_{2p}$ of order $2p$, we consider the problem of computing a basis of the associated order in each Hopf Galois structure and the module structure of the ring of integers $\mathcal{O}_L$. We solve the case in which $L/\mathbb{Q}_p$ is not totally ramified and present a practical method which provides a complete answer for the cases $p=3$ and $p=5$. We see that within this family of dihedral extensions, the ring of integers 
is always free
over the associated orders in the different Hopf Galois structures.
\end{abstract}

\section{Introduction}

A finite extension of fields $L/K$ is said to be Hopf Galois if there is a $K$-Hopf algebra $H$ and a $K$-linear action 
$\langle\  ,\  \rangle:H\otimes_KL\to L$ which endows $L$ with an $H$-module algebra structure $H\to \End_K(L)$, such that the linear map $j\colon L\otimes_K H\longrightarrow\End_K(L)$ is an isomorphism. In that case, the pair formed by the Hopf algebra and the Hopf action is said to be a Hopf Galois structure of $L/K$. This notion generalizes the one of Galois extension: if $L/K$ is Galois with group $G$, the pair formed by the group algebra $K[G]$ and its Galois action on $L$ provides a Hopf Galois structure of $L/K$.

The Greither-Pareigis theorem  (see \cite{greitherpareigis} and \cite[Theorem 6.8]{childs}) characterizes all Hopf Galois structures of a separable degree $n$ extension $L/K$. Let us call  $\widetilde{L}$ the normal closure of $L$, $G=\Gal(\widetilde{L}/K)$, $G'=\Gal(\widetilde{L}/L)$, $X=G/G'$ and $\lambda:G\to \mathrm{Perm}(X)$ the left action on cosets. 

\begin{thm}[Greither-Pareigis] Hopf Galois structures of $L/K$ are in one-to-one correspondence with regular subgroups of $\mathrm{Perm}(X)$ normalized by $\lambda(G)$. Moreover, if $N$ is some such subgroup, the corresponding Hopf Galois structure is given by the $K$-Hopf algebra $\widetilde{L}[N]^G$ and its action over $L$ defined by $$\left\langle \sum_{i=1}^rc_in_i, x\right\rangle=\sum_{i=1}^rc_in_i^{-1}(\,\overline{1_G}\,)(x)$$
\end{thm}

The type of a Hopf Galois structure of $L/K$ is defined as the isomorphism class of $N$ as an abstract group. If the extension 
$L/K$ is Galois the classical structure corresponds to $N=\rho(G)$, the centralizer of $\lambda(G)$ in $\mathrm{Perm}(G)$, $H=K[G]$ and the classical Galois action of $G$ in $L$ extended by linearity. If $G$ is non-commutative we have at least one other Hopf Galois structure of the same type, corresponding to $N=\lambda(G)$, which is called the canonical non-classical structure.

For $L/K$ a Galois extension of $p$-adic fields with Galois group $G$, starting from H. Leopoldt \cite{leop} the ring of integers 
$\OL$ is studied as a module over its associated order $\mathfrak{A}_{K[G]}=\{h\in K[G]: h\OL\subseteq \OL\}$. 
The main question that arises is to determine if $\OL$ is free as  $\mathfrak{A}_{K[G]}$-module. In the context of Hopf Galois theory the question generalizes in a natural way, namely we want to determine if $\OL$ is free as  $\mathfrak{A}_{H}$-module with 
$$\mathfrak{A}_H=\{h\in H  : h\OL\subseteq \OL\}=\{h\in H  :\  \langle h,x\rangle \in \OL\quad \forall x\in\OL\},$$ where $H$ is an arbitrary Hopf Galois structure of $L/K$. 
The classical Galois action is just one of the different Hopf actions that we can have on the field $L$. We will denote the associated order in the classical Galois structure by $\mathfrak{A}_{L/K}$. 
 
In this paper we consider extensions $L/\Qp$ with Galois group $G$ isomorphic to the dihedral group $D_{2p}$ of order $2p$.
In this case Hopf Galois structures can be only of type $D_{2p}$ (dihedral) or $C_{2p}$ (cyclic), and we know from \cite[Theorem 6.2]{byottpq} that there are two Hopf Galois structures of dihedral type and $p$ of cyclic type. The structures of dihedral type are the classical Galois structure and the canonical non-classical one. As for the cyclic type, it is a split $C_p\times C_2$ type for the dihedral group $D_{2p}=C_p\rtimes C_2$ and we know from \cite{cresporiovela} that all such structures are induced, namely obtained from Hopf Galois structures of disjoint subextensions.

For these extensions, we aim to study the problem of determining the freeness of $\mathcal{O}_L$ as module over the associated order $\mathfrak{A}_{L/K}$ in a Hopf Galois structure. This problem is actually solved for Hopf Galois structures of dihedral type. Indeed, if $L/K$ is not totally ramified, then its inertia group is cyclic and \cite[Corollaire after Théoréme 3]{berge2} gives that $\mathcal{O}_L$ is $\mathfrak{A}_{L/K}$-free. The remainder of the cases are covered by \cite[Proposition 7]{berge}. Moreover, Truman proved that $\mathcal{O}_L$ is free over its associated order in the classical Galois structure if and only if so is over the one in the canonical non-classical Hopf Galois structure (see \cite[Theorem 1.1]{truman}). Therefore, $\mathcal{O}_L$ is free over its associated order in the Hopf Galois structures of dihedral type.

We shall study the module structure of $\mathcal{O}_L$ over its associated order in Hopf Galois structures of cyclic type. The main tools are the techniques introduced in \cite{gilrio}, which are reformulated in Section \ref{secredmethod}. The idea is that if one knows explicitly how a given Hopf Galois structure $H$ acts on an integral basis of $L$, then one can compute explicitly the generalized module index $[\mathcal{O}_L:\langle\mathfrak{A}_H,\beta\rangle]_{\mathcal{O}_K}$ (see Proposition \ref{criteriafreeness} and the proceeding remark). In Section \ref{sechgstructures} we describe the Hopf Galois structures on a dihedral degree $2p$ extension $L/K$, and in Section \ref{secdihedral} we give details about the arithmetic for the case $K=\mathbb{Q}_p$. Namely, $L/\mathbb{Q}_p$ has a generating polynomial which is among a class of defining polynomials discovered by Amano, which we subsequently call Amano polynomials, and we use this fact to determine the discriminant and the chain of ramification groups.

To determine the Hopf Galois module structure of the ring of integers for the ones of cyclic type, we use that those Hopf Galois structures are induced. In Section \ref{secnontotcase} we use some general results concerning the relation between freeness and tensored Hopf Galois structures to obtain the following:

\begin{thm} Let $L/\Qp$ be a non-totally ramified dihedral degree $2p$ extension of $p$-adic fields. 
\begin{itemize}
    \item[1.] If $E/\Qp$ is a degree $p$ subextension of $L/\Qp$ and $H_1$ is its unique Hopf Galois structure, then $\ROE$ is $\mathfrak{A}_{H_1}$-free.
    \item[2.] If $H$ is a Hopf Galois structure of $L/\Qp$, then $\mathcal{O}_L$ is $\mathfrak{A}_H$-free.
\end{itemize}
\end{thm}

For the rest of the cases, we need the techniques of \cite{gilrio}. In Section \ref{secppart} we study the case of a degree $p$ subextension $E/\mathbb{Q}_p$ of a dihedral degree $2p$ extension. We present some considerations that arise from applying the reduction method to those Hopf Galois structures, which in turn gives a procedure which is feasible for $p\in\{3,5\}$. We will obtain the following:

\begin{thm} Let $p\in\{3,5\}$. If $E/\mathbb{Q}_p$ is a separable degree $p$ extension with dihedral degree $2p$ normal closure, then $\mathcal{O}_E$ is free over the associated order of its unique Hopf Galois structure. Moreover, one can find explicitly a basis of that associated order.
\end{thm}

As for dihedral degree $2p$ extensions themselves, in Section \ref{sechopfcyclic} we glue the results obtained from their quadratic and degree $p$ subextensions to present a procedure to study the associated order of a Hopf Galois structure of cyclic type, as well as the module structure of $\mathcal{O}_L$. In this case, we will prove the following:

\begin{thm} Let $p\in\{3,5\}$. If $L/\Qp$ is a dihedral degree $2p$ extensions of $p$-adic fields and $H$ is a Hopf Galois structure on $L/\Qp$ of cyclic type then $\mathcal{O}_L$ is free over the associated order in $\mathfrak{A}_H$. Moreover, one can find explicitly a basis of that associated order.
\end{thm}

Most of the computations in this paper are shown in the line of the development of the contents. The ones that are sophisticated enough have been carried out with Maple, and some of them are considerably bulky and shown in the Appendix \ref{appendix}.

\section{The reduction method revisited}\label{secredmethod}

The main tool in this paper consists in the set of techniques introduced in the paper \cite{gilrio}, where the authors established a general method to obtain a basis of the associated order in every Hopf Galois structure and provided a necessary and sufficient condition to determine whether or not the ring of integers is free over the associated order. In this section we summarize the concepts and results obtained in that article, and we also introduce new notions that will be useful for our purposes.

The standard situation throughout the paper is the following: 
\begin{itemize}
    \item $K$ is the fraction field of a principal ideal domain $\mathcal{O}_K$.
    \item $L$ is a finite separable field extension of $K$.
    \item $\OL$ is the integral closure of $\mathcal{O}_K$ in $L$.
\end{itemize}

\subsection{The matrix of the action}\label{firstsubsection}

The key to determine the associated order $\mathfrak{A}_H$ is to study the Hopf action $$\langle\ ,\ \rangle\colon H\otimes_K L\longrightarrow L$$ as $K$-linear map. To this end, we fix $K$-bases $W=\{w_i\}_{i=1}^n$ of $H$ and $B=\{\gamma_j\}_{j=1}^n$ of $L$.

We consider the matrix $G(H_W,L_B)=(\langle w_i,\gamma_j\rangle)_{i,j=1}^n\in\mathcal{M}_n(L)$, which we call \textit{the Gram matrix of the action}. This is the matrix that in \cite{gilrio} we represented by a table in the various examples throughout the paper. If $B'$ is another $K$-basis of $L$, it is easy to check that $$G(H_W,L_{B'})=G(H_W,L_B)P_B^{B'},$$where $P_B^{B'}$ is the change of basis 
matrix.

On the other hand, \textit{the matrix of the action of $H$ on $L$} is defined as the matrix $M(H_W,L_B)$ whose columns are the entries of the matrices representing $w_i$ by means of the representation $\rho_H\colon H\longrightarrow\mathrm{End}_K(L)$ (see \cite[Definition 3.1]{gilrio}). In other words, $M(H_W,L_B)$ is the matrix of the linear map $\rho_H$ where in $H$ we consider the $K$-basis $W$ and in $\mathrm{End}_K(L)$ we consider the $K$-basis $\Phi=\{\varphi_i\}_{i=1}^{n^2}$ defined as follows: For every $1\leq i\leq n^2$, there are $1\leq k,j\leq n$ such that $i=k+(j-1)n$. Then, let $\varphi_i$ be the map that sends $\gamma_j$ to $\gamma_k$ and the other $\gamma_l$ to $0$.

Explicitely, if $\langle w_i,\gamma_j\rangle=\sum_{k=1}^n m_{ij}^{(k)}\gamma_k$ with $m_{ij}^{(k)}\in K$, then $m_{ij}^{(k)}$ is the entry corresponding to the $(n(j-1)+k)$-th row and $i$-th column. Then, the matrix of the action can be written as $$M(H_W,L_B)=\begin{pmatrix}
\\[-1ex]
M_1(H_W,L_B) \\[1ex] 
\hline\\[-1ex]
\cdots \\[1ex]
\hline \\[-1ex]
M_{n}(H_W,L_B)\\
\\[-1ex]
\end{pmatrix},$$ where $M_j(H_W,L_B)=(m_{ij}^{(k)})_{k,j=1}^n$ is the matrix of the linear map 
$\langle\ , \gamma_j\rangle : H  \longrightarrow  L $ and will be called \textit{the $j$-th block of the matrix of the action} in the sequel. There is a formula for a change of basis of the $H$ in terms of these blocks: if $W'$ is another $K$-basis of $H$, then $$M_j(H_{W'},L_B)=M_j(H_W,L_B)P_W^{W'}.$$

The matrix $M(H_W,L_B)$ can be obtained from $G(H_W,L_B)$ as follows. 
Given $1\leq j\leq n$, we consider the $j$-th column $C_j$ of $G(H_W,L_B)$. 
Then, $M_j(H,L)$ is the matrix arising 
from taking as $i$-th column the coordinates of 
the $i$-th component of $C_j$ with respect to the basis $B$.

When the bases of $H$ and $L$ are implicit in the context, we simply write $M(H,L)$ for the matrix of the action.

\subsection{Determining a basis of $\mathfrak{A}_H$}

Let $L/K$ be an $H$-Galois extension of fields such that $K$ is the fraction field of a PID $\mathcal{O}_K$. In this section we present the steps to carry out the reduction method and obtain a basis of the associated order in $H$.

\begin{prop} Assume that $B$ is an $\mathcal{O}_K$-basis of $\mathcal{O}_L$ (i.e, an integral basis of $L$). Let $h=\sum_{i=1}^{n}h_iw_i\in H$, $h_i\in K$. Then, $$h\in\mathfrak{A}_H \iff  M(H,L)\begin{pmatrix}h_1 \\ \vdots \\ h_{n}\end{pmatrix}\in\mathcal{O}_K^{n^2}.$$
\end{prop}
\begin{proof}
See \cite[Theorem 3.3]{gilrio}.
\end{proof}

To compute a basis of $\mathfrak{A}_H$, we reduce $M(H,L)$ in such a way that the property of the previous proposition is preserved. It is possible to accomplish this because of the Bézout property, which at the same time can be performed because $\mathcal{O}_K$ is a PID. That is:

\begin{thm} There are matrices $U\in\mathrm{GL}_{n^2}(\mathcal{O}_K)$ and $D\in\mathrm{GL}_n(K)$ such that 
$$
U\, M(H,L)=\begin{pmatrix}
D\\ \hline \\[-2ex] O
\end{pmatrix}.
$$
\end{thm}
\begin{proof}
Since $\mathcal{O}_K$ is a PID, every pair of elements in $\mathcal{O}_K$ satisfies the Bézout identity. Then, 
we can apply general methods of matrix reduction, and even assume that $D$ is upper triangular (see \cite[Theorem 3.5]{kaplansky}).
\end{proof}

The matrix $D$ of the previous statement will be called \textit{a reduced matrix} of $M(H,L)$.
Such a matrix provides a basis of the associated order.

\begin{thm}\label{teobasisassocorder} Let $D$ be a reduced matrix of $M(H,L)$ and call $D^{-1}=(d_{ij})_{i,j=1}^n$. Then, the elements $$v_i=\sum_{l=1}^{n}d_{li}w_l,\quad 0\leq i\leq n-1$$ form an $\mathcal{O}_K$-basis of $\mathfrak{A}_H$.
\end{thm}
\begin{proof}
See \cite[Theorem 3.5]{gilrio}.
\end{proof}

\subsection{The Hermite normal form over a PID}\label{subsecthermite}

A reduced matrix $D$ of $M(H,L)$ is not unique: the left multiplication of $D$ by any elementary matrix that preserves $\mathcal{O}_L$ is another reduced matrix of $M(H,L)$. But in practice we will take a particular reduced matrix: the Hermite normal form. Note that in general $M(H,L)$ does not have coefficients in $\mathcal{O}_K$. In that case, we can always write $M(H,L)=dM$ with $d\in K$ and $M\in\mathcal{M}_n(\mathcal{O}_K)$, and we define the Hermite normal form of $M(H,L)$ as $dH$, where $H$ is the Hermite normal form of $M$. Since the usual definition of Hermite normal form is for matrices with integer coefficients, we use the definition in \cite[Section 5.2]{adkinsweintraub}, which is valid for matrices with coefficients in a PID. Namely:

\begin{defi}\label{defihermite} Let $R$ be a PID, $P$ a complete set of non-associates in $R$, and for every $a\in R$, let $P(a)$ be a complete set of residues modulo $a$. Let $M=(m_{ij})\in\mathcal{M}_{m\times n}(R)$ be a non-zero matrix. We will say that $M$ is in Hermite normal form if there exists an integer $1\leq r\leq m$ such that:
\begin{itemize}
    \item[1.] Given $1\leq i\leq r$, the $i$-th row of $M$ is non-zero, and given $r+1\leq i\leq m$, the $i$-th row of $M$ is zero.
    \item[2.] There is a sequence of integer numbers $1\leq n_1<\cdots<n_r\leq m$ such that for every $1\leq i\leq r$:
    \begin{itemize}
        \item Given $j<n_i$, $m_{ij}=0$.
        \item $m_{i,n_i}\in P-\{0\}$.
        \item Given $1\leq j<i$, $m_{j,n_i}\in P(m_{i,n_i})$
    \end{itemize}
\end{itemize}
\end{defi}

The Hermite normal form of a matrix with coefficients in a PID always exists and is unique (\cite[Chapter 5, Theorems 2.9 and 2.13]{adkinsweintraub}). 
In this article we will work with $R=\Zp$. For this ring, $\{p^n\}_{n=1}^{\infty}$ is a complete set of non-associates. Thus, the Hermite normal form of $M(H,L)$ will be an upper triangular matrix such that:

\begin{itemize}
    \item The entries of the diagonal are a non-negative power of $p$. 
    \item The entries above an entry $1$ of the diagonal are $0$ and the entries above an entry $p^k$ are representatives of a coset mod $p^k$.
\end{itemize}

\subsection{Freeness over the associated order}\label{subsectfreeness}

Let $L/K$ be an $H$-Galois extension such that $K$ is the fraction field of a PID $\mathcal{O}_K$. For any $K$-bases $W$ of $H$ and $B$ of $L$, the matrix of the action associated to an element $\beta=\sum_{j=1}^n\beta_j\gamma_j\in L$ is defined as the matrix of $\langle\ , \beta\rangle : H  \longrightarrow  L $, namely
$$M_{\beta}(H_W,L_B)=\sum_{j=1}^n\beta_jM_j(H_W,L_B).$$

Assume that $B$ is an integral basis of $L$. Let $D$ be a reduced matrix of $M(H_W,L_B)$ and let $V$ be the $\mathcal{O}_K$-basis of $\mathfrak{A}_H$ provided by Theorem \ref{teobasisassocorder}. Since $\mathfrak{A}_H$ acts on $\mathcal{O}_L$, the matrix $M(H_V,L_B)$ has coefficients in $\mathcal{O}_K$. It is its invertibility as an element of $\mathcal{M}_n(\mathcal{O}_K)$ what determines whether $\beta$ is a free generator of $\mathcal{O}_L$ as $\mathfrak{A}_H$-module.

\begin{prop}\label{criteriafreeness} An element $\beta\in\mathcal{O}_L$ is an $\mathfrak{A}_H$-free generator of $\mathcal{O}_L$ if and only if the associated matrix $M_{\beta}(H_V,L_B)$ is unimodular (i.e. it belongs to $\mathrm{GL}_n(\mathcal{O}_K)$).
\end{prop}
\begin{proof}
See \cite[Proposition 4.2]{gilrio}.
\end{proof}

\begin{rmk}\normalfont Whenever $\beta$ is a primitive element of $L/K$, the determinant of the matrix $M_{\beta}(H_V,L_B)$ is actually the generalized module index $[\mathcal{O}_L:\langle\mathfrak{A}_H,\beta\rangle]_{\mathcal{O}_K}$ (or more accurately, it is the ideal generated by the determinant). Indeed, both $\mathcal{O}_L$ and $\langle\mathfrak{A}_H,\beta\rangle$ are $\mathcal{O}_K$-orders in $L$ and the canonical map $\varphi\colon\langle\mathfrak{A}_H,\beta\rangle\longrightarrow\mathcal{O}_L$ has matrix $M_{\beta}(H_V,L_B)$, where we fix the basis $\{\langle v_i,\beta\rangle\}_{i=1}^n$ in $\langle\mathfrak{A}_H,\beta\rangle$ and the basis $B$ in $\mathcal{O}_L$. Then, Proposition \ref{criteriafreeness} means that $\mathcal{O}_L=\langle\mathfrak{A}_H,\beta\rangle$ if and only if $[\mathcal{O}_L:\langle\mathfrak{A}_H,\beta\rangle]_{\mathcal{O}_K}$ is trivial.
\end{rmk}

Since $D=P_W^V$, using the change basis formula for each $M_j(H_W,L_B)$ in Section \ref{firstsubsection}, it is immediate that $$M_\beta(H_V,L_B)=M_\beta(H_W,L_B)D^{-1}.$$ Then, the criterion of the previous result is satisfied if and only if $$D_{\beta}(H_W,L_B)\in\mathrm{det}(D)\mathcal{O}_K^*,$$ where $D_{\beta}(H_W,L_B)=\mathrm{det}(M_{\beta}(H_W,L_B))$. Now, assume that $\mathcal{O}_K$ is a discrete valuation ring with valuation $v_K$. Then, the previous condition is equivalent to $$v_K(D_{\beta}(H_W,L_B))=v_K(\mathrm{det}(D)).$$

\begin{prop} The number $v_K(\mathrm{det}(D))$ does not depend on the reduced matrix chosen.
\end{prop}  
\begin{proof}
Let us call $M(H,L)=M(H_W,L_B)$ for convenience.

Since $M(H,L)$ has rank $n$, its Hermite normal form is 
$
\begin{pmatrix}
M\\ \hline \\[-2ex] O
\end{pmatrix}
$ for a certain matrix $M\in GL_n(K)$ in echelon form. Let $D$ be a reduced matrix of $M(H,L)$. By definition of reduced matrix, the matrices $M(H,L)$ and $
\begin{pmatrix}
D\\ \hline \\[-2ex] O
\end{pmatrix}
$ are equal up to left multiplication by a unimodular matrix (left equivalent according to \cite[Definition 2.1]{adkinsweintraub}) and the Hermite normal form is unique (see \cite[Theorem 2.13]{adkinsweintraub}), so they have the same Hermite normal form $
\begin{pmatrix}
M\\ \hline \\[-2ex] O
\end{pmatrix}
$. 

We claim that $M$ is the Hermite normal form of $D$. Indeed, if $M'$ is the Hermite normal form of $D$, then there is a unimodular matrix $U\in\mathrm{GL}_n(\mathcal{O}_K)$ such that $UD=M'$. Then $\left(\begin{array}{c|c}
    U & 0 \\ \hline
    0 & I_{n^2-n}
\end{array}\right)\in\mathcal{M}_{n^2}(\mathcal{O}_K)$ is a unimodular matrix because its determinant is $\mathrm{det}(U)\in\mathcal{O}_K^*$, and satisfies $$\left(\begin{array}{c|c}
    U & O' \\ \hline
    O & I_{n^2-n}
\end{array}\right)\begin{pmatrix}
D \\ \hline \\[-2ex] O
\end{pmatrix}=\begin{pmatrix}
M' \\ \hline \\[-2ex] O
\end{pmatrix},$$ where $O'$ is the zero matrix in $\mathcal{M}_{n\times(n^2-n)}(K)$. By the uniqueness of the Hermite normal form, $
\begin{pmatrix}
M\\ \hline \\[-2ex] O
\end{pmatrix}=
\begin{pmatrix}
M'\\ \hline \\[-2ex] O
\end{pmatrix}
$, that is, $M=M'$. Moreover, we deduce that $UD=M$.

This proves that for every reduced matrix $D$ there is a unimodular matrix carrying $
\begin{pmatrix}
D\\ \hline \\[-2ex] O
\end{pmatrix}
$ to $
\begin{pmatrix}
M\\ \hline \\[-2ex] O
\end{pmatrix}
$ that has its first $n\times n$ block $U$ unimodular, and the equality $UD=M$ holds for this block. Then $v_K(\mathrm{det}(D))=v_K(\mathrm{det}(M))$ for every reduced matrix $D$.
\end{proof}

This leads to the following definition:

\begin{defi} Let $W$ be a $K$-basis of $H$. The index of the Hopf Galois structure $(H,\langle,\rangle)$ is defined as $$I_W(H,L)=v_K(\mathrm{det}(D)),$$ where $D$ is a reduced matrix of $M(H_W,L)$.
\end{defi}

If $W'$ is another $K$-basis of $H$, then $$I_{W'}(H,L)=I_W(H,L)+v_K(\mathrm{det}(P_W^{W'})).$$ With this notation, we can rewrite Proposition \ref{criteriafreeness} in a more convenient way, without requiring the basis $V$:

\begin{prop} An element $\beta\in\mathcal{O}_L$ is an $\mathfrak{A}_H$-free generator of $\mathcal{O}_L$ if and only if $$v_K(\mathrm{det}(M_{\beta}(H_W,L_B)))=I_W(H,L).$$
\end{prop}

Note that this new condition does not depend of the basis of $H$. Indeed, if $W'$ is another $K$-basis of $H$, then $M_{\beta}(H_{W'},L_B)=M_{\beta}(H_W,L_B)P_W^{W'}$, so $$v_K(\mathrm{det}(M_{\beta}(H_{W'},L_B)))=\mathrm{det}(M_{\beta}(H_W,L_B)))+v_K(\mathrm{det}(P_W^{W'})),$$ and joining this with the equality of indexes above, we obtain that $$v_K(\mathrm{det}(M_{\beta}(H_W,L_B)))=I_W(H,L) \iff v_K(\mathrm{det}(M_{\beta}(H_{W'},L_B)))=I_{W'}(H,L)$$

In the case $K=\Q_p$, we will choose $D$ to be the Hermite normal form of $M(H,L)$ (over $\Zp$), and then $I_W(H,L)$ will be the number of $p$'s that appear in the diagonal of $D$ (with multiplicities).

\subsection{An example: Quadratic extensions}\label{sectquadr}

Let $L/K$ be a degree $2$ separable extension such that $K$ is the fraction field of a PID $\mathcal{O}_K$ and $\mathrm{char}(K)\neq2$. Then $L/K$ is Galois with Galois group of the form $G=\{1,\sigma\}$, and $L=K(z)$ with $\sigma(z)=-z$. Moreover, $H=K[G]$ is the unique Hopf Galois structure of $L/K$. In \cite[Example 3.7]{gilrio} we applied the reduction method to compute a basis of the associated order $\mathfrak{A}_{H}$, with the $K$-basis $\{\mathrm{Id},\sigma\}$ of $H$ and the basis $\{1,z\}$ of $L$. Let us assume that the last basis is integral; in particular $z\in\mathcal{O}_L$. The Gram matrix arising from this choice is $$G(H,L)=\begin{pmatrix}
1 & z \\
1 & -z
\end{pmatrix},$$ and the computation of $M(H,L)$ follows easily. We obtained the reduced matrix $D=\begin{pmatrix}
1 & 1 \\
0 & 2 
\end{pmatrix}$ which gives the basis $\left\lbrace\mathrm{Id},\frac{-\mathrm{Id}+\sigma}{2}\right\rbrace.$ The matrix $D$ is the Hermite normal form of $M(H,L)$ when we consider coefficients in $\mathbb{Z}$. If we take instead a field $K$ such that $2$ is invertible in $\mathcal{O}_K$ (for instance, $K=\mathbb{Q}_p$ with $p\geq3$ prime), the Hermite normal form turns out to be the identity matrix. This gives the $\mathcal{O}_K$-basis $\left\lbrace\mathrm{Id},\sigma\right\rbrace$ of $\mathfrak{A}_{H}$. It follows then that $\mathfrak{A}_H=\mathcal{O}_K[G]$.

Regarding the freeness of $\mathcal{O}_L$ over $\mathfrak{A}_{H}$, for $\delta=\delta_1+\delta_2z$, we have $$M_{\delta}(H,L)=\delta_1
\begin{pmatrix}
1 & 1 \\
0 & 0
\end{pmatrix}+
\delta_2
\begin{pmatrix}
0 & 0 \\
1 & -1
\end{pmatrix}=
\begin{pmatrix}
\delta_1 & \delta_1 \\ 
\delta_2 & -\delta_2
\end{pmatrix},$$ with determinant $-2\delta_1\delta_2$. Then, this determinant is invertible in $\mathcal{O}_K$ if and only if so is $\delta_1\delta_2$. Thus, $\mathcal{O}_L$ is $\mathfrak{A}_H$-free and every $\delta=\delta_1+\delta_2z$ with $\delta_1\delta_2\in\mathcal{O}_L^*$ is a generator (and no other element of $\mathcal{O}_L$ is). For example, $\delta=1+z$ is such a free generator.

We summarize the results obtained in the following:

\begin{thm}\label{teoquadreasy}
Let $L/K$ be a quadratic extension such that $\mathrm{char}(K)\neq2$ and $2\in\mathcal{O}_K^*$, and let $G$ be its Galois group. Assume that there is $z\in L$ with $z\notin K$ and $z^2\in K$ such that $\{1,z\}$ is an integral basis of $L$. Then, $\mathfrak{A}_{L/K}=\mathcal{O}_K[G]$. Moreover, the elements that generate $\mathcal{O}_L$ as $\mathfrak{A}_{L/K}$-module are the elements $\delta=\delta_1+\delta_2z\in\mathcal{O}_L$ such that $\delta_1\delta_2\in\mathcal{O}_L^*$.
\end{thm}

\section{Description of the Hopf Galois structures}\label{sechgstructures}

In this section, we take $L/K$ to be a dihedral degree $2p$ extension of fields with $\mathrm{char}(K)\neq2$. Let $G=\mathrm{Gal}(L/K)\cong D_{2p}$. From now on, we establish the following presentation of $G$: $$G=\langle r,s\mid r^p=s^2=1,\,sr=r^{p-1}s\rangle.$$
The problem of counting and describing the Hopf Galois structures of such an extension was solved completely by Byott in his paper \cite{byottpq}, in which he actually studies Galois extension of degree $pq$, where $q$ is a prime number dividing $p-1$. We recover the degree $2p$ extensions by choosing $q=2$. By the Greither-Pareigis theorem, the Hopf Galois structures of $L/K$ are in one-to-one correspondence with regular subgroups $N$ of $\mathrm{Perm}(G)$ normalized by $\lambda(G)$,  where $\lambda$ is the left regular representation of $G$. Since the order of such a subgroup is $2p$, the Hopf Galois structures can be divided in:

\begin{itemize}
    \item[1.] Dihedral type: The group $N$ is isomorphic to $D_{2p}$. Since $G$ is not abelian, there are two different Hopf Galois structures of this type: the classical Galois structure, given by the case $N=\rho(G)$ (where $\rho\colon G\longrightarrow\mathrm{Perm}(G)$ is the right regular representation), and the canonical non-classical Galois structure, in which case $N=\lambda(G)$.
    \item[2.] Cyclic type: The group $N$ is isomorphic to $C_{2p}$. If we call $\mu=\lambda(r)$ and $\eta_i=\rho(r^is)$, then there are $p$ Hopf Galois structures of this type, given by $$N_i=\langle \mu,\eta_i\rangle,\,0\leq i\leq p-1.$$
\end{itemize}

By \cite[Theorem 6.2]{byottpq}, those are all the Hopf Galois structures of $L/K$. As already mentioned, we are especially interested in the Hopf Galois structures of cyclic type. In the remainder of the section we describe the Hopf algebra of those Hopf Galois structures. By \cite[Theorem 9]{cresporiovela}, the cyclic Hopf Galois structures of $L/K$ are the induced ones. Then, they can be described in terms of Hopf Galois structures of subextensions of $L/K$.

 The lattice of intermediate fields of $L/K$ is in one-to-one inclusion-reversing correspondence with the lattice of subgroups of $G$. The group $G$ has a unique order $p$ subgroup $J=\langle r\rangle$, which corresponds to the extension $F/K$, and $p$ order $2$ subgroups $G_i'=\langle r^is\rangle$, $i\in\{1,...,p\}$. Then, there are $p$ subextensions of $L/K$ of degree $p$. None of the extensions $E/K$ are Galois because $G_i'$ is not a normal subgroup of $G$ for every $i$, but they all have $J$ as normal complement. Thus, every extension $E/K$ is an almost classically Galois extension. Moreover, by \cite[Theorem 2]{byottuniqueness}, each extension $E/K$ has a unique Hopf Galois structure.

By \cite[Proposition 5.5]{gilrio}, the induced Hopf Galois structures correspond to all possible decompositions of $G$ as semidirect product. Those decompositions are $G=J\rtimes G_i'$, $1\leq i\leq p$. Namely, the induced Hopf Galois structures are of the form $$H=H_1^{(i)}\otimes H_2,$$ where, for every $i$, $H_1^{(i)}$ is the Hopf Galois structure of $L^{G_i'}/K$ and $H_2$ is the Hopf Galois structure of $F/K$. Thus, in order to describe the induced Hopf Galois structures, it is enough to characterize those of the subextensions of $L/K$.

The easiest case is the Hopf Galois structure $H_2$ of the quadratic extension $F/K$. Indeed, this is the classical Galois structure of $F/K$, which is the $K$-group algebra of the Galois group. Now, the Galois group of this extension is generated by the restriction of any automorphism $r^is$ to $F$ (note that since $r$ fixes $F$, that restriction does not depend on $i$). Then, $H_2=K[\eta_i]$, where $\eta_i=\rho(r^is)$ and $\rho$ is the right translation map of $\mathrm{Gal}(F/K)$ in its group of permutations.

Let us fix a degree $p$ subextension $E/K$ of $L/K$ and let $H_1$ be its unique Hopf Galois structure. Then $G'\coloneqq\mathrm{Gal}(L/E)=\langle r^is\rangle$ for some $1\leq i\leq p$. By the Greither-Pareigis theorem, $H_1=L[N_1]^G$, where $N_1$ is the unique regular subgroup of $\mathrm{Perm}(G/G')$ normalized by $\overline{\lambda}(G)$, where $\overline{\lambda}\colon G\longrightarrow\mathrm{Perm}(X)$ is the left translation map of $G$ in $\mathrm{Perm}(X)$.  One can easily check that $N_1=\overline{\lambda}(J)$ satisfies the required properties. Let us call $\overline{\mu}=\overline{\lambda}(r)$, which can be expressed as the permutation $(\overline{1_G},\overline{r},\dots,\overline{r^{p-1}})$ of the quotient set $G/G'$. This element is the generator of $N_1$, that is, $$N_1=\langle\overline{\mu}\rangle=\{\overline{\mathrm{Id}},\overline{\mu},\dots\overline{\mu}^{p-1}\}.$$ Let us determine the Hopf algebra $H_1=L[N_1]^G$ of this Hopf Galois structure.

\begin{prop}\label{propbasisppart} Let $E/K$ be a degree $p$ subextension of a dihedral degree $2p$ extension $L/K$. The unique Hopf Galois structure $H_1$ of $E/K$ has a $K$-basis given by the identity $w_1=\mathrm{Id}$ and the elements 
$$
    w_{1+i}=z(\overline{\mu}^{\,i}-\overline{\mu}^{\,-i}), \quad
    w_{\frac{p+1}{2}+i}=\overline{\mu}^{\,i}+\overline{\mu}^{\,-i},
$$
where $1\leq i\leq\frac{p-1}{2}$ and $z\in L$ is such that $z\notin K$ and $z^2\in K$.
\end{prop}
\begin{proof}
Let $x\in H_1$. Since $H_1\subset L[N_1]$, there are elements $a_i\in L$ with $0\leq i\leq p-1$ such that $$x=\sum_{i=0}^{p-1}a_i\overline{\mu}^{,i}.$$ The action of $G$ on $N_1$ is given by $r(\overline{\mu})=\overline{\mu},\quad(\overline{\mu})=\overline{\mu}^{\,-1}.$ Now, since $x\in H$, it is fixed by the action of $G$ on $H_1$. Then,
\begin{equation*}
    \begin{split}
        x=r(x)&=\sum_{i=0}^{p-1}r(a_i)\overline{\mu}^{\,i}, \\
        x=s(x)&=\sum_{i=0}^{p-1}s(a_i)\overline{\mu}^{\,-i}=\sum_{i=0}^{p-1}s(a_{p-i})\overline{\mu}^{\,i},
    \end{split}
\end{equation*} where in the last line we consider the subscripts mod $p$. The first equality gives that $r(a_i)=a_i$ for all $0\leq i\leq p-1$, so $a_i\in L^{\langle r\rangle}=F$. On the other hand, the second one yields $s(a_i)=a_{p-i}$ for every $i$. Since $s(a_0)=a_0$, we get $a_0\in K$. \\

Now, $a_i\in F$ for $1\leq i\leq p$, so there are $a_i^{(1)},a_i^{(2)}\in K$ such that $a_i=a_i^{(1)}+a_i^{(2)}z$. For those values of $i$, $s(a_i)=a_i^{(1)}-a_i^{(2)}z$, but also $s(a_i)=a_{p-i}$, so $a_{p-i}=a_i^{(1)}-a_i^{(2)}z$. Then, \begin{equation}\label{descrpelemppart}
    \begin{split}
        x&=a_0\mathrm{Id}+\sum_{i=0}^{\frac{p-1}{2}}(a_i^{(1)}+a_i^{(2)}z)\overline{\mu}^{\,i}+\sum_{i=0}^{\frac{p-1}{2}}(a_i^{(1)}-a_i^{(2)}z)\overline{\mu}^{\,i}\eta\\&=a_0\mathrm{Id}+\sum_{i=0}^{\frac{p-1}{2}}a_i^{(1)}(\overline{\mu}^{\,i}+\overline{\mu}^{\,-i})+\sum_{i=0}^{\frac{p-1}{2}}a_i^{(2)}z(\overline{\mu}^{\,i}-\overline{\mu}^{,-i}).
    \end{split}
\end{equation}

We have obtained a set of generators and therefore a $K$-basis, since $H_1$ has dimension $p$. 
\end{proof}

\begin{rmk}\normalfont Let $\lambda\colon G\longrightarrow\mathrm{Perm}(G)$ be the left regular representation of $G$. Since the powers of $\mu=\lambda(r)$ factorize through $G'$ as permutations of $G$, the groups $\overline{\lambda}(J)$ and $\lambda(J)$ can be identified by establishing $\overline{\lambda}(r^i)=\lambda(r^i)$, and in particular, $\overline{\mu}=\mu$. Hence, from now on, we will take the elements of the basis of $H_1$ in the statement above with the powers of $\mu$ instead of $\overline{\mu}$.
\end{rmk}

The elements $w_k$ for $k$ even (resp. odd) can be described as linear combinations of even (resp. odd) powers of $w_2=z(\mu-\mu^{-1})$, so they generate $H_1$ as $K$-algebra. That is, $H_1=K[z(\mu-\mu^{-1})]$. Then, one can conclude the following:

\begin{coro}\label{corobasisinduced} The induced Hopf Galois structures of $H$ are $$H^{(i)}=K[z(\mu-\mu^{-1}),\eta_i],\,1\leq i\leq p.$$ With the notation of Proposition \ref{propbasisppart}, a $K$-basis of $H^{(i)}$ is $$\{w_1,\dots,w_p,w_1\eta,\dots,w_p\eta\}.$$
\end{coro}

This finishes the description of all Hopf Galois structures of $L/K$. Note that it coincides with the expression of the cyclic Hopf Galois structures shown in \cite[Section 6]{KKTU}.

\section{The arithmetic of the extension}\label{secdihedral}

From \cite{awtreyedwards} we have an explicit description of the degree $2p$ dihedral $p$-adic fields:
\begin{enumerate}
\item If $p = 3$, there are six non-isomorphic cubic extensions of $\Q_3$ whose normal closures have Galois group isomorphic to 
$D_{6}$. The polynomials defining these extensions can be chosen to be the Amano polynomials 
$$
x^3+3,\quad
x^3+12,\quad
x^3+21,\quad
x^3+3x+3,\quad
x^3+6x+3,\quad
x^3+3x^2+3.
$$
The inertia subgroups for these fields are all dihedral of order $6$ with the exception of the field defined by 
$x^3 + 3x^2 + 3$, whose inertia subgroup is cyclic of order 3.
\item If $p > 3$, there are three non-isomorphic degree $p$ extensions of $\Q_p$ whose normal closures have Galois group isomorphic to 
$D_{2p}$. These extensions are defined by the polynomials 
$$
x^p+px^{p-1}+p,\quad
x^p+2px^{\frac{p-1}2}+p,\quad
x^p+(p-2)p x^{\frac{p-1}2}+p.
$$
The inertia subgroup of the field defined by $x^p + px^{p-1} + p$ is cyclic of order $p$. 
The other two fields have inertia subgroup equal to $D_{2 p}$.
\end{enumerate}

Let us call $f$ any of these polynomials and let $L$ be the splitting field of $f$ over $\mathbb{Q}_p$, so $L/\mathbb{Q}_p$ is a dihedral degree $2p$ extension. Note that the $p$ degree $p$ subextensions $E/\mathbb{Q}_p$ of $L/\mathbb{Q}_p$ correspond to $E=\mathbb{Q}_p(\alpha)$ with $\alpha$ running through the roots of $f$. Let us fix such an intermediate field $E=\Qp(\alpha)$. Since the minimal polynomial of $\alpha$ is $p$-Eisenstein, the extension $E/\Qp$ is totally ramified and its ring of integers is $\ROE=\Zp[\alpha]$. Hence, $L/\Qp$ is always wildly ramified. Recall that since 
$L/\Qp$ is 
dihedral, $L/\Qp$ has a unique quadratic subextension $F/\Qp$. We write $F=\Qp(z)$ with $z^2\in\Qp$.

\subsection{Discriminants and ramification}

We can use Ore conditions from \cite{ore} to determine the discriminants of the totally ramified extensions $E/\Qp$. Translated to our case, we obtain:

\begin{prop}
Given $j_0\in \Z$, there exist totally ramified extensions $E/\Q_p$ of degree $p$ and discriminant $p^{p+j_0-1}$ (up to multiplication by an invertible element of $\mathbb{Z}_p$) if and only if
$1\leq j_0\leq p$. Moreover:
\begin{enumerate}
\item For $j_0=p$, an Eisenstein polynomial $f(x)=x^p+\displaystyle\sum_{i=0}^{p-1}f_ix^i \in \Z_p[x]$ has discriminant ${p}^{2p-1}$ if 
$v_{ p}(f_i) \ge 2$ for  $1 \le i \le p -1.$
\item For $0< j_0< p$, an Eisenstein polynomial $f(x)=x^p+\displaystyle\sum_{i=0}^{p-1}f_ix^i \in \Z_p[x]$ has discriminant ${p}^{p+j_0-1}$ if
$v_{ p}(f_i) \ge 2$ for  $1 \le i < p-1$, $i\neq j_0$, and $v_{p}(f_{j_0} ) =1$.
\end{enumerate}
\end{prop}

The statement and the proof in its more general form can be consulted also in \cite[Proposition 3.1 and Lemma 5.1]{pauliroblot}.

Applying the result above to our degree $p$ extensions, we obtain the following data:
\begin{center}
\begin{tabular}{ c ||l | c | c }
& Polynomial & $j_0$ &$v_p(d_{E/\Q_p})$ \\
\hline \hline
$p=3$ & $x^3+3a$ \quad ($a=1,4,7$)  & 3 & $5$ \\
\hline
$p\ge 3$ &$x^p+(p-2)p x^{\frac{p-1}2}+p$ &$\dfrac{(p-1)}2$&$\frac{3(p-1)}2$\\
&$x^p+2px^{\frac{p-1}2}+p$&$\dfrac{(p-1)}2$& $\frac{3(p-1)}2$ \\
& $x^p+px^{p-1}+p$&$p-1$& $2(p-1)$ \\
\end{tabular}
\end{center}

Note that the unused values of $j_0$ between $1$ and $p$ correspond to other degree $p$ extensions of $\mathbb{Q}_p$ with a different Galois closure.

Now, taking into account the relative discriminant formula
$$
d_{L/\Q_p}=N_{E/\Q_p}(d_{L/E}) \cdot d_{E/\Q_p}^2
$$
we obtain the discriminant of the dihedral extension. 
If $L/E$ is unramified, then $d_{L/\Q_p}= d_{E/\Q_p}^2$ and 
otherwise, $d_{L/\Q_p}= p\cdot d_{E/\Q_p}^2$.


On the other hand, let us denote $G_i$ the $i$-th ramification group of the extension $L/\Qp$. If the inertia group is $C_p$, then 
$$
v_p(d_{L/\Q_p})=4(p-1)=2( \sum_{i\ge 0} (|G_i|-1))=2( (p-1)+(p-1)+\dots),
$$ whereas for the totally ramified cases we have
$$
v_p(d_{L/\Q_p})=3p-2= \sum (|G_i|-1)=( 2p-1)+(p-1)+\dots
$$
except for the radical extensions, where the equality is $11=5+2+\dots$.
Therefore, for all the non-radical extensions we have that the second ramification group
 $G_2$ is trivial. These extensions are called \emph{weakly ramified}. 

\begin{center}
\begin{tabular}{ c ||l |  c | c| c}
& Polynomial  & $v_p(d_{E/\Q_p})$ & $v_p(d_{L/\Q_p})$ & $G_0\supseteq G_1\supseteq\dots $ \\
\hline \hline
$p=3$ &$x^3+3a$\quad ($a=1,4,7$)   & $5$ & $11$&$S_3\supseteq C_3\supseteq C_3\supseteq C_3\supseteq 1$\\
\hline
$p\ge 3$ &$x^p+(p-2)p x^{\frac{p-1}2}+p$ & $\frac{3(p-1)}2$& $3p-2$& $D_{2p}\supseteq C_p\supseteq 1$\\
&$x^p+2px^{\frac{p-1}2}+p$& $\frac{3(p-1)}2$ & $3p-2$& $D_{2p}\supseteq C_p\supseteq 1$\\
& $x^p+px^{p-1}+p$& $2(p-1)$ & $4(p-1)$& $C_p\supseteq C_p\supseteq 1$\\
\end{tabular}
\end{center}

Then, the classification of dihedral extensions $L/\mathbb{Q}_p$ of degree $2p$ of $\mathbb{Q}_p$ can be presented in a more convenient way:

\begin{itemize}
    \item[1.] If $L/\mathbb{Q}_p$ is weakly ramified, then $L$ is the splitting field over $\mathbb{Q}_p$ of one of the polynomials $$x^p+2px^{\frac{p-1}{2}}+p,\quad x^p+p(p-2)x^{\frac{p-1}{2}}+p,\quad x^p+px^{p-1}+p.$$
    \item[2.] Otherwise, $p=3$ and $L$ is the splitting field over $\mathbb{Q}_3$ of one of the polynomials $$x^3+3,\quad x^3+12,\quad x^3+21.$$
\end{itemize}

\begin{rmk}\normalfont The radical cases are the unique ones for which $t=3$, which corresponds to $L/\mathbb{Q}_p$ having maximal ramification. Since $L/\mathbb{Q}_p$ is wildly ramified, this situation is a particular case of \cite[Corollaire after Proposition 6]{berge}. Namely, it gives that for a wildly ramified extension with dihedral inertia, almost maximal ramification is equivalent to $t_1=3$ and $p=3$.
\end{rmk}


\subsection{Basis and primitive elements}\label{subsectprimit}

Recall the notation $E=\Qp(\alpha)$ and $F=\Qp(z)$. Since $E/\Qp$ and $F/\Qp$ are linearly disjoint
for the extension $L/\Qp$, we can consider the tensor basis  
$$ 
\{1, z, \alpha, \alpha z, \dots , \alpha^{p-1}, \alpha^{p-1}z\}
$$
When $F/\Qp$ is unramified these extensions are arithmetically disjoint and this is also a $\Z_p$-basis for the ring of integers 
 $\OL$. When this is not the case, this tensor basis is not well suited to search for free generators of $\OL$. Since we are dealing with totally ramified extensions, we can look for a primitive element $\gamma$  which is also a uniformizing element of $\OL$. Under these conditions, the powers of $\gamma$ form a $\Q_p$- basis of $L$ and also a 
$\Z_p$- basis of $\OL$. 

\begin{prop}\label{propunif} Let $p\ge 3$ be a prime.
Let $f$ be a polynomial of degree $p$ having totally ramified dihedral normal closure of degree $2p$.
Let $\alpha$ be a root of $f$ in the algebraic closure of $\Qp$ and $E=\Qp(\alpha)$. Let $L/\Qp$ be the normal closure of $f$ and $F/\Qp$ be its quadratic subextension. Write $F=\Qp(z)$ with $v_F(z)=1$. Then, $$\gamma=\dfrac{z}{\alpha^{\frac{p-1}2}}
$$ 
has $v_L(\gamma)=1$.
\end{prop}
Indeed, we have 
$$v_F(z)=v_E(\alpha) = 1 \implies v_L(z) = p,\  v_L(\alpha) = 2 \implies v_L(\gamma) = p - 2\dfrac{(p - 1)}2 =1.$$

\section{The non-totally ramified case}\label{secnontotcase}

Assume that the dihedral degree $2p$ extension $L/\mathbb{Q}_p$ of $p$-adic fields is not totally ramified. We have seen that any degree $p$ subextension $E/\mathbb{Q}_p$ is totally ramified, so this is the same as saying that the quadratic subextension $F/\mathbb{Q}_p$ is unramified. Then, the extensions $E/\mathbb{Q}_p$ and $F/\mathbb{Q}_p$ are arithmetically disjoint. We can exploit this fact to solve completely the problem of the freeness for this case.

\subsection{The extension $L/F$}

Following the notations fixed in Section \ref{secdihedral}, we now focus on the extension $L/F$. It is a Galois extension with Galois group $J=\langle r\rangle$, the unique order $p$ subgroup of $G$. That is, it is a cyclic degree $p$ extension. Therefore, the classical Galois structure $F[\lambda(J)]$ is its unique Hopf Galois structure. The problem of the Galois module structure of the integers rings of such an extension was completely solved by F. Bertrandias, J.P. Bertrandias and M.J. Ferton (see \cite{bertbertfert} for the proof and \cite[Theorem 3.4]{thomas} for the statement):

\begin{thm} Let $K$ be a finite extension of $\mathbb{Q}_p$ and let $L$ be a totally ramified degree $p$ extension of $K$. Let $t$ be the ramification number of $L/\Qp$. Let $\mathfrak{A}_{L/K}$ be the associated order in the classical Galois structure of $L/K$.
\begin{itemize}
    \item[1.] If $p$ divides $t$, then $\mathcal{O}_L$ is $\mathfrak{A}_{L/K}$-free.
    \item[2.] Otherwise, we write $t=pk+a$ for the euclidean division of $t$ by $p$, and we have:
    \begin{itemize}
        \item[(i)] If $0\leq t<\frac{pe(K/\mathbb{Q}_p)}{p-1}-1$, $\mathcal{O}_L$ is $\mathfrak{A}_{L/K}$-free if and only if $a$ divides $p-1$.
        \item[(ii)] If $\frac{pe(K/\mathbb{Q}_p)}{p-1}-1\leq t\leq \frac{pe(K/\mathbb{Q}_p)}{p-1}$, $\mathcal{O}_L$ is $\mathfrak{A}_{L/K}$-free if and only if the length of the expansion of $\frac{t}{p}$ as continuous fraction is at most $4$.
    \end{itemize}
\end{itemize}
\end{thm}

We apply this result to our situation with the cyclic degree $p$ extension $L/F$, which is totally ramified. The ramification number can be obtained from the table in Section \ref{secdihedral}: it is $t=1$ when $p>3$ and $t\in\{1,3\}$ when $p=3$, depending on whether or not the extension is weakly ramified. 

For $p>3$, we have that $p$ does not divide $t$, and since $t<p$, we have $a=t$. If the defining polynomial is one of the first two, then $F/\mathbb{Q}_p$ is ramified and $e(F/\mathbb{Q}_p)=2$, thus $$\frac{pe(F/\mathbb{Q}_p)}{p-1}-1=\frac{2p}{p-1}-1=\frac{p+1}{p-1},$$ and $t<\frac{p+1}{p-1}$. Since $t$ divides $p-1$, $\mathcal{O}_L$ is $\mathfrak{A}_{L/F}$-free. 
For the last polynomial, $F/\mathbb{Q}_p$ is unramified and $e(F/\mathbb{Q}_p)=1$, so $$\frac{pe(F/\mathbb{Q}_p)}{p-1}-1=\frac{p}{p-1}-1=\frac{1}{p-1}$$ and then $\frac{1}{p-1}<t<\frac{p}{p-1}$. The expansion of $\frac{1}{p}$ is trivial, so $\mathcal{O}_L$ is $\mathfrak{A}_{L/F}$-free. 

Now assume $p=3$. For the radical cases, we have that $t=3$ and $e(F/\mathbb{Q}_3)=2$. Hence $$\frac{3e(F/\mathbb{Q}_3)}{2}-1=\frac{6}{2}-1=2,$$ and then $t=\frac{3e(F/\mathbb{Q}_3)}{2}$. The expansion of $\frac{3}{3}=1$ is again trivial, so $\mathfrak{A}_{L/F}$-free. Now, if $L/\mathbb{Q}_3$ is weakly ramified, for the totally ramified cases we have $t=1$ and $e(F/\mathbb{Q}_3)=2$, so $t<\frac{3e(F/\mathbb{Q}_3)}{2}-1$ and as it divides $2$, $\mathcal{O}_L$ is $\mathfrak{A}_{L/F}$-free. Finally, if $L/\mathbb{Q}_3$ is weakly ramified and it is not totally ramified (the last polynomial), then $t=e(F/\mathbb{Q}_3)=1$, so $$\frac{3e(F/\mathbb{Q}_3)}{2}-1=\frac{3}{2}-1=\frac{1}{2}$$ and $t>\frac{1}{2}$. Since the expansion of $\frac{1}{3}$ is trivial, $\mathcal{O}_L$ is $\mathfrak{A}_{L/F}$-free.

In summary, we obtain:

\begin{coro}\label{corofreenesstensorf} Let $L/\mathbb{Q}_p$ be a dihedral degree $2p$ extension and let $F$ be the unique subfield of $L$ which is quadratic over $\mathbb{Q}_p$. Then, $\mathcal{O}_L$ is $\mathfrak{A}_{L/F}$-free.
\end{coro}

\subsection{Descent to the extension $E/\mathbb{Q}_p$}

Under the condition that $E/\mathbb{Q}_p$ and $F/\mathbb{Q}_p$ are arithmetically disjoint (which we recall is equivalent to $L/\mathbb{Q}_p$ not being totally ramified), since $H_1\otimes_KF$ is the classical Galois structure of $L/F$, we can easily prove the $\mathfrak{A}_{H_1}$-freeness of $\mathcal{O}_E$ implies the $\mathfrak{A}_{L/F}$-freeness of $\mathcal{O}_F$ (see \cite[Corollary 5.19]{gilrio}). Actually, Lettl proved that the converse of the corresponding result for the classical Galois structure holds (see \cite[Proposition 1]{lettl}). His proof can be perfectly adapted to prove the converse also for arbitrary Hopf Galois structures: the key step is the application of the Krull-Schmidt-Azumaya theorem (see \cite[(6.12)]{curtisreiner}), which also is possible in this case because $\mathfrak{A}_{E/\mathbb{Q}_p}$ is $\mathbb{Z}_p$-finitely generated and $\mathbb{Z}_p$ is a complete discrete valuation ring. This combined with Corollary \ref{corofreenesstensorf} gives the following:

\begin{coro}\label{corofreenessnontot} Let $E/\mathbb{Q}_p$ be a separable degree $p$ extension of $p$-adic fields with dihedral degree $2p$ normal closure and let $H_1$ be its unique Hopf Galois structure. Then, $\mathcal{O}_E$ is $\mathfrak{A}_{H_1}$-free.
\end{coro}

\subsection{The degree $2p$ extension}

If $L/\mathbb{Q}_p$ is not totally ramified, the arithmetic disjointness of $E/\mathbb{Q}_p$ and $F/\mathbb{Q}_p$ allows us to apply \cite[Theorem 5.11]{gilrio} and \cite[Theorem 5.16]{gilrio}, obtaining:

\begin{coro}\label{corounramified} Let $L/\mathbb{Q}_p$ be a non-totally ramified degree $2p$ extension of $p$-adic fields and let $H=H_1\otimes_{\mathbb{Q}_p}H_2$ be a Hopf Galois structure of cyclic type on $L/\mathbb{Q}_p$. Then:
\begin{itemize}
    \item[1.] $\mathfrak{A}_H=\mathfrak{A}_{H_1}\otimes_{\Zp}\mathfrak{A}_{H_2}$.
    \item[2.] $\mathcal{O}_L$ is $\mathfrak{A}_H$-free.
\end{itemize}
\end{coro}

This result closes completely the problem of the freeness, but it does not give the exact form of the associated order as long as we lack a description of $\mathfrak{A}_{H_1}$. In the following sections we obtain such a description and study the freeness in the totally ramified cases.

\section{The case of a separable degree $p$ extension with dihedral Galois closure}\label{secppart}

Recall that the Hopf Galois structures of cyclic type of a dihedral degree $2p$ extension $L/\mathbb{Q}_p$ are the induced ones. When we work with induced Hopf Galois structures we know that $H=H_1\otimes H_2$, where $H_1$ is a Hopf Galois structure of $E/\Q_p$ and $H_2$ is a Hopf Galois structure of $F/\Q_p$. In order to determine a basis of $H$, we have worked with the factors and joined the information obtained. 

We can proceed in an analog way to study the associated order $\mathfrak{A}_H$ and the structure of $\mathcal{O}_L$ as $\mathfrak{A}_H$-module. Namely, we consider the same problem for the subextensions of a dihedral degree $2p$ extension, i.e. the quadratic subextension and the degree $p$ ones, and use this information to study the dihedral degree $2p$ case. As for the quadratic extension $F/\mathbb{Q}_p$, we know by Theorem \ref{teoquadreasy} that $\mathfrak{A}_{H_2}=\mathcal{O}_K[\eta]$, where, with the notation of Section \ref{sechgstructures}, $\eta=\eta_i$ for some $i$. In this section we consider the $p$-part, that is, we work with a separable degree $p$ extension $E/\mathbb{Q}_p$ whose Galois closure is a dihedral degree $2p$ extension.

\subsection{The general method}

We know that $E=\mathbb{Q}_p(\alpha)$ for a root $\alpha$ of one of the polynomials $f$ at the beginning of Section \ref{secdihedral}.

The extension $E/\Q_p$ has a unique Hopf Galois structure $H_1$ with a $\Q_p$-basis as in Proposition \ref{propbasisppart}.
We will need to assume that $\{1,z\}$ is an integral basis of $F=\mathbb{Q}_p(z)$.
On the other hand, we know that $\{1,\alpha,\dots,\alpha^{p-1}\}$ is a $\Zp$-basis of $\mathcal{O}_E$. In these bases the action of $H_1$ over $E$ can be expressed in terms of Lucas sequences. 

Since we have to deal with conjugates of $\alpha$ and their powers, first we group the roots of $f$ in a convenient way. Namely, since $\alpha$ is a root of $f$, there is a polynomial $f_1\in E[x]$ such that $f(x)=(x-\alpha)f_1(x).$ 
Now, since $L$ is the normal closure of $E/\Q_p$ and $[L:E]=2$, $$f_1(x)=\prod_{i=1}^{\frac{p-1}{2}}P_i(x)=\prod_{i=1}^{\frac{p-1}{2}}(x^2-A_ix+B_i),$$ 
with $A_i,B_i\in E$ for every $1\leq i\leq\frac{p-1}{2}$. Let us call $d_i=A_i^2-4B_i\in E$ the discriminant of $P_i$. Then, the roots of $P_i$ are
\begin{align*}
    \alpha_{2,i}=\frac{A_i+\sqrt{d_i}}{2}, && \alpha_{3,i}=\frac{A_i-\sqrt{d_i}}{2}.
\end{align*}
We can assume without loss of generality that $r^{-i}(\alpha)=\alpha_{2,i}$ for every $1\leq i\leq \frac{p-1}{2}$ (otherwise we would reorder the polynomials $P_i$). Then, $$\alpha_{3,i}=s(\alpha_{2,i})=sr^{-i}(\alpha)=r^is(\alpha)=r^i(\alpha).$$ Namely, the roots of $P_i$ are $r^{-i}(\alpha)$ and $r^i(\alpha)$. To compute the action of the elements of the basis of $H_1$, we have to deal with sums and differences of each pair of roots of $P_i$. The suitable tool to deal with such objects are Lucas sequences.

\begin{defi} Let $K$ be a field and let $A,B\in K$. The Lucas sequences of first kind $U_j(A,B)$ and of second kind $V_j(A,B)$ for the parameters $A,B$ are defined by the expressions \begin{align*}
    &U_0(A,B)=0,\\
    &U_1(A,B)=1,\\
    &U_j(A,B)=AU_{j-1}(A,B)-BU_{j-2}(A,B),\,j\geq2,
\end{align*}
\begin{align*}
    &V_0(A,B)=2,\\
    &V_1(A,B)=A,\\
    &V_j(A,B)=AV_{j-1}(A,B)-BV_{j-2}(A,B),\,j\geq2.
\end{align*}
\end{defi}

The characteristic polynomial of $U_j(A,B)$ and $V_j(A,B)$ is defined as $P(x)=x^2-Ax+B$. The relationship of the sequences with the roots of the polynomial is given by the following result,
whose proof is straightforward by induction on $j$.

\begin{prop} Assume that the discriminant $d=A^2-4B$ of $P$ is non-zero and let $\alpha_2=\frac{A+\sqrt{d}}{2}$ and $\alpha_3=\frac{A-\sqrt{d}}{2}$ be its roots. Then, $$U_j(A,B)=\frac{\alpha_2^j-\alpha_3^j}{\sqrt{d}}, \quad V_j(A,B)=\alpha_2^j+\alpha_3^j.$$
\end{prop}

We compute the action of $H_1$ over $E$ as follows:
\begin{thm} Let us consider the basis $W=\{w_i\}_{i=1}^p$ of $H_1$ as before and the integral basis $B=\{\alpha^i\}_{i=0}^{p-1}$ of $E$. Call $d_i$ the discriminant of $P_i$. Then, for every $1\leq i\leq\frac{p-1}{2}$ and every $0\leq j\leq p-1$, \begin{align*}
    &\langle w_1,\alpha^j\rangle=\alpha^j,
    &\langle w_{1+i},\alpha^j\rangle=U_j(A_i,B_i)\sqrt{d_i}z,
    &&\langle w_{\frac{p+1}{2}+i},\alpha^j\rangle=V_j(A_i,B_i).
\end{align*}
\end{thm}
\begin{proof}
The first equality is trivial since $w_1=\mathrm{Id}$. Let $1\leq i\leq\frac{p-1}{2}$ and $0\leq j\leq p-1$. Then,
\begin{equation*}
    \begin{split}
       \langle w_{1+i},\alpha^j\rangle&=\langle(\mu^i-\mu^{-i})z,\alpha^j\rangle
       =(r^{-i}(\alpha^j)-r^i(\alpha^j))z\\
       &=(\alpha_{2,i}^j-\alpha_{3,i}^j)z=U_j(A_i,B_i)\sqrt{d_i}z \\
        \langle w_{\frac{p+1}{2}+i},\alpha^j\rangle&=\langle\mu^i+\mu^{-i},\alpha^j\rangle
        =r^{-i}(\alpha^j)+r^i(\alpha^j)\\&=\alpha_{2,i}^j+\alpha_{3,i}^j=V_j(A_i,B_i)
    \end{split}
\end{equation*}
\end{proof}

\begin{coro}\label{corogramppart} Let us call $U_i(P_i)=U_i(A_i,B_i)$ and $V_i(P_i)=V_i(A_i,B_i)$ for every $1\leq i\leq\frac{p-1}{2}$. The Gram matrix of $H_1$ is $$G(H_1,E)=\begin{pmatrix}
1 & \alpha & \cdots & \alpha^{p-1} \\
U_0(P_1)\sqrt{d_1}z & U_1(P_1)\sqrt{d_1}z & \cdots & U_{p-1}(P_1)\sqrt{d_1}z \\
U_0(P_2)\sqrt{d_2}z & U_1(P_2)\sqrt{d_2}z & \cdots & U_{p-1}(P_2)\sqrt{d_2}z \\
\vdots & \vdots & \ddots & \vdots \\
U_0(P_{\frac{p-1}{2}})\sqrt{d_{\frac{p-1}{2}}}z &
U_1(P_{\frac{p-1}{2}})\sqrt{d_{\frac{p-1}{2}}}z & \cdots & U_{p-1}(P_{\frac{p-1}{2}})\sqrt{d_{\frac{p-1}{2}}}z \\
V_0(P_1) & V_1(P_1) & \cdots & V_{p-1}(P_1) \\
V_0(P_2) & V_1(P_2) & \cdots & V_{p-1}(P_2) \\
\vdots & \vdots & \ddots & \vdots \\
V_0(P_{\frac{p-1}{2}}) & V_1(P_{\frac{p-1}{2}}) & \cdots & V_{p-1}(P_{\frac{p-1}{2}})
\end{pmatrix}$$
\end{coro}

Note that the previous result implies that $\sqrt{d_i}z\in E$ for every $1\leq i\leq\frac{p-1}{2}$, and recall that $z$ can be any element such that $\{1,z\}$ is an integral basis of $F$. In practice, what we will usually do is to choose $z$ after computing $\sqrt{d_1}$, so that the expression of $\sqrt{d_1}z$ is convenient enough.

All the previous considerations lead to the following method to compute $G(H_1,E)$:
\begin{itemize}
    \item[1.] Factorize the polynomial $f$ in terms of a root $\alpha$ to compute the polynomials $P_i\in E[x]$, $1\leq i\leq\frac{p-1}{2}$.
    \item[2.] For every $1\leq i\leq\frac{p-1}{2}$,  compute the square root of $d_i$.
    \item[3.] Determine the entries of $G(H_1,E)$ following Corollary \ref{corogramppart}.
\end{itemize}

Once $G(H_1,E)$ is computed, we can determine $M(H_1,E)$ from its entries and use the reduction method to obtain a basis of $\mathfrak{A}_{H_1}$ and determine the freeness of $\mathcal{O}_E$ as $\mathfrak{A}_{H_1}$-module. We still consider the non-totally ramified cases, although we already know that $\mathcal{O}_E$ is indeed $\mathfrak{A}_{H_1}$-free by Corollary \ref{corofreenessnontot}, because that result does not give an explicit generator, and this procedure does.

\subsection{The case $p=3$}

Let $E/\mathbb{Q}_3$ be a separable degree $3$ extension of $3$-adic fields with dihedral degree $6$ normal closure. We know by Section \ref{secdihedral} that $E$ is generated by a root of one of the polynomials $$x^3+3,\quad x^3+12,\quad x^3+21,$$ $$x^3+3x+3,\quad x^3+6x+3,\quad x^3+3x^2+3.$$ We divide these polynomials in three groups. The first three are of the form $x^3+3a$ with $a\in\{1,4,7\}$, and these are the radical cases. The next two polynomials may be expressed by $x^3+3ax+3$ with $a\in\{1,2\}$, while the sixth polynomial is the unique one for which $L/\mathbb{Q}_p$ is not totally ramified. From now on, these will be called the first and second group and the singular case, respectively.

\subsubsection{The action on the $3$-part}

The extension $E/\Q_3$ has a unique Hopf Galois structure $H_1$ with $\Q_3$-basis $$w_1=\mathrm{Id},\quad w_2=z(\mu-\mu^{-1}),\quad w_3=\mu+\mu^{-1}$$ where $\mu=\lambda(r)$ and $z$ is any element of $L$ with square in $\Qp$. Let $f$ denote one of the previous polynomials and let us fix the root $\alpha$ of $f$ such that $E=\Q_3(\alpha)$. We know that $$f(x)=(x-\alpha)P(x)=x^2-Ax+B,$$ with $A,B\in E$. Let $d$ be the discriminant of $f$. According to \ref{corogramppart}, $$G(H_1,E)=\begin{pmatrix}
1 & \alpha & \alpha^2 \\
0 & \sqrt{d}z & A\sqrt{d}z \\
2 & A & A^2-2B
\end{pmatrix},$$ and the matter is to determine $\sqrt{d}z$. 
We obtain:
\begin{center}
\begin{tabular}{ l ||c | c | c }
Polynomial & $A$ & $B$ & $d$ \\
\hline \hline
$x^3+3a$, $a\in\{1,4,7\}$  & $-\alpha$ & $\alpha^2$ & $-3\alpha^2$ \\
\hline
$x^3+3ax+3$, $a\in\{1,2\}$ & $-\alpha$ & $\alpha^2+3a$ & $-3\alpha^2-12a$ \\ \hline
$x^3+3x^2+3$ & $-(\alpha+3)$ & $\alpha^2+3\alpha$ & $-3\alpha^2-6\alpha+9$ \\
\end{tabular}
\end{center}

Note that $d$ is the square of the difference between the two roots of $P$. Hence, in order to determine its square-root, we are free to choose the sign, because each one corresponds to a different ordering of the roots.

For the radical cases, the situation is quite easy. Indeed, if we call $z=\sqrt{-3}$, we have $\sqrt{d}=\alpha z$ and $F=\mathbb{Q}_3(\sqrt{-3})$. Then, $\sqrt{d}z=-3\alpha$.

Let us consider polynomials of the second group. If $a=1$, we need to compute the square root of $-3\alpha^2-12$ in $L$. 
Solving a system of equations, we get 
$$-3\alpha^2-12=-{\frac {12}{13\,}}(\alpha^2-\frac{3}{2}\alpha+2)^2=-{\frac {3}{13}}(2\alpha^2-3\alpha+4)^2.$$ Then, $$\sqrt{d}=\sqrt{\frac{-3}{13}}(2\alpha^2-3\alpha+4),$$ 
(the other choice of sign corresponds to exchange the roots of the quadratic polynomial). Let $z=\sqrt{-39}$. Since $13\equiv 1\,(\mathrm{mod}\,3)$, 
$F=\mathbb{Q}_3(z)=\mathbb{Q}_3(\sqrt{-3})$. Then, $$\sqrt{d}z=-6\alpha^2+9\alpha-12.$$ If $a=2$, we find 
$\sqrt{d}=\sqrt{\dfrac{3}{-41}}(-4\alpha^2+3\alpha-16).$ Let us choose $z=\sqrt{-123}$. 
Then, $F=\mathbb{Q}_3(z)=\mathbb{Q}_3(\sqrt{3})$ and $$\sqrt{d}z=-12\alpha^2+9\alpha-48.$$
Finally, for the singular case, $\sqrt{d}=\sqrt{-\dfrac{1}{7}}(-2\alpha^2-9\alpha-3),$ $z=\sqrt{-7}$ and $$\sqrt{d}z=2\alpha^2+9\alpha+3.$$

Now, it is easy to fill the whole matrix $G(H_1,E)$ at each case. We obtain:

\begin{prop} The Gram matrix of the action of $H_1$ on $E$ is given by:
\begin{itemize}
    \item[1.] If $f(x)=x^3+3a$, $a\in\{1,4,7\}$, 
    $$G(H_1,E)=\begin{pmatrix}
    1 & \alpha & \alpha^2 \\
    0 & -3\alpha & 3\alpha^2 \\
    2 & -\alpha & -\alpha^2
    \end{pmatrix}.$$
    \item[2.] If $f(x)=x^3+3ax+3$, $a\in\{1,2\}$, $$G(H_1,E)=\begin{pmatrix}
    1 & \alpha & \alpha^2 \\
    0 & -6a\alpha^2+9\alpha-12a^2 & -9\alpha^2-6a^2\alpha-18a \\
    2 & -\alpha & -\alpha^2-6a
    \end{pmatrix}.$$
    \item[3.] If $f(x)=x^3+3x^2+3$, 
    $$G(H_1,E)=\begin{pmatrix}
    1 & \alpha & \alpha^2 \\
    0 & 2\alpha^2+9\alpha+3 & -9\alpha^2-30\alpha-3 \\
    2 & -\alpha-3 & -\alpha^2+9
    \end{pmatrix}$$
\end{itemize}
\end{prop}

\subsubsection{Basis of $\mathfrak{A}_{H_1}$}\label{basisassocorder3part}

The matrix of the action $M(H_1,E)$ is built from $G(H_1,E)$ as explained in Section \ref{firstsubsection}. According to the cases of the last proposition, the matrix $M(H_1,E)$ is \begin{align*}
    \begin{pmatrix}
1 & 0 & 2 \\
0 & 0 & 0 \\
0 & 0 & 0 \\
0 & 0 & 0 \\
1 & -3 & -1 \\
0 & 0 & 0 \\
0 & 0 & 0 \\
0 & 0 & 0 \\
1 & 3 & -1
\end{pmatrix}, && \begin{pmatrix}
1 & 0 & 2 \\
0 & 0 & 0 \\
0 & 0 & 0 \\
0 & -12a^2 & 0 \\
1 & 9 & -1 \\
0 & -6a & 0 \\
0 & -18a & -6a \\
0 & -6a^2 & 0 \\
1 & -9 & -1
\end{pmatrix}, && \begin{pmatrix}
1 & 0 & 2 \\
0 & 0 & 0 \\
0 & 0 & 0 \\
0 & 3 & -3 \\
1 & 9 & -1 \\
0 & 2 & 0 \\
0 & -3 & 9 \\
0 & 30 & 0 \\
1 & -9 & -1
\end{pmatrix},
\end{align*} respectively. Since the powers of $\alpha$ form an integral basis of $E$, a reduced matrix of $M(H_1,E)$ gives a basis of $\mathfrak{A}_{H_1}$. We compute the Hermite normal form of these matrices over $\mathbb{Z}_3$ 
We obtain the matrix $$\begin{pmatrix}
1 & 0 & -1 \\
0 & 3 & 0 \\
0 & 0 & 3
\end{pmatrix}$$ for the first and second group of polynomials. Then, $\left\lbrace w_1,\dfrac{w_2}{3},\dfrac{w_1+w_3}{3}\right\rbrace$ is a $\mathbb{Z}_3$-basis of $\mathfrak{A}_{H_1}$ in all these cases. For the singular case, the Hermite normal form is $$\begin{pmatrix}
1 & 0 & -1 \\
0 & 1 & 0 \\
0 & 0 & 3
\end{pmatrix}.$$ Then, $\left\lbrace w_1,w_2,\dfrac{w_1+w_3}{3}\right\rbrace$ is a $\mathbb{Z}_3$-basis of $\mathfrak{A}_{H_1}$.

\subsubsection{Freeness over the associated order}

Let $\epsilon=\sum_{i=1}^3\epsilon_i\alpha^{i-1}\in\mathcal{O}_E$. For each polynomial, we compute the determinant $D_{\epsilon}(H_1,E)$ of $M_{\epsilon}(H_1,E)$ in terms of the coordinates $\epsilon_1,\,\epsilon_2,\,\epsilon_3$. By Proposition \ref{criteriafreeness}, $\epsilon$ is a free generator of $\mathcal{O}_E$ as $\mathfrak{A}_{H_1}$-module if and only if $I_W(H_1,E)=v_3(D_{\epsilon}(H_1,E)$. We know the index $I_W(H_1,E)$ at each case from the previous section. Then, it suffices to find $(\epsilon_1,\epsilon_2,\epsilon_3)$ fulfilling the previous equality.

\begin{center}
\begin{tabular}{ l ||c | c  }
Polynomial & $I_W(H,L)$ & $D_{\epsilon}(H_1,E)$  \\
\hline \hline
$x^3+3a$, $a\in\{1,4,7\}$  & $2$ & $18\epsilon_1\epsilon_2\epsilon_3$  \\
\hline
$x^3+3ax+3$, $a\in\{1,2\}$ & $2$ & $-18(a\epsilon_2^2+3\epsilon_2\epsilon_3-a^2\epsilon_3^2)(\epsilon_1-2a\epsilon_3)$  \\ \hline
$x^3+3x^2+3$ & $1$ & $6(\epsilon_2^2-9\epsilon_2\epsilon_3+15\epsilon_3^2)(\epsilon_1-\epsilon_2+3\epsilon_3)$  \\
\end{tabular}
\end{center}

The equality $I_W(H_1,E)=v_3(D_{\epsilon}(H_1,E))$ is achieved:

\begin{itemize}
    \item In the radical cases, for $\epsilon=\epsilon_1+\epsilon_2\alpha+\epsilon_3\alpha^2$.
    \item In the second case, for $\epsilon=1+\alpha$.
    \item In the singular case, for $\epsilon=2-\alpha$.
\end{itemize}

Hence, we obtain:

\begin{prop} If $E/\mathbb{Q}_3$ is a separable degree $3$ extension with dihedral degree $6$ Galois closure and $H_1$ is its unique Hopf Galois structure, then $\mathcal{O}_E$ is $\mathfrak{A}_{H_1}$-free.
\end{prop}

\subsection{The case $p=5$}

Let $E/\mathbb{Q}_5$ be a separable degree $5$ extension of $5$-adic fields with dihedral degree $10$ normal closure. In this case, $E=\mathbb{Q}_5(\alpha)$ for a root $\alpha$ of one of the polynomials $$x^5+15x^2+5,\quad x^5+10x^2+5,\quad x^5+5x^4+5.$$ Again, we call $f$ the polynomial among the previous ones that defines $E/\mathbb{Q}_5$. The procedure is essentially the same as in the case $p=3$, with an important difference: given a root $\alpha$ of $f$, the polynomial $f$ is not irreducible over $\mathbb{Q}(\alpha)$, so the quadratic polynomials $P_1$, $P_2$ in the decomposition of the form $$f(x)=(x-\alpha)P_1(x)P_2(x)$$ have coefficients in $\mathbb{Q}_5(\alpha)$ but not in $\mathbb{Q}(\alpha)$, and we are not able to compute them explicitly. 

In order to overcome this difficulty, we work with a polynomial generating the same field and whose decomposition over $\mathbb{Q}_5(\alpha)$ is also a decomposition over $\mathbb{Q}(\alpha)$. For the $5$-adic field $E$ generated by each of the polynomials $f$ above, we can find in the database \cite{lmfdb} a polynomial $g$ defining a number field whose completion at $5$ is $E$.

\begin{center}
\begin{tabular}{ c ||c | c }
$f$ & $g$ & LMFDB identifier  \\
\hline \hline
$x^5+15x^2+5$  & $x^5-15x^3-10x^2+75x+30$ & 5.1.23765625.1 \\
\hline
$x^5+10x^2+5$ & $x^5-35x^2+50x+20$ & 5.1.34515625.1 \\ \hline
$x^5+5x^4+5$ & $x^5+10x^4+50x^3+125x^2+150x+60$ & 5.1.3515625.1 \\
\end{tabular}
\end{center}

\subsubsection{The action on the $5$-part}\label{sectgrammatrix5}

For each polynomial $g$ defining each of the non-normal degree $5$ extensions of $\mathbb{Q}_5$ above, we find the decomposition $$g(x)=(x-\alpha)P_1(x)P_2(x)$$ for a previously fixed root $\alpha$ of $g$. Now, let us write $P_i(x)=x^2-A_ix+B_i$, $i\in\{1,2\}$ and let $d_i=A_i^2-4B_i$. By Corollary \ref{corogramppart}, $$G(H_1,E)=\begin{pmatrix}
1 & \alpha & \alpha^2 &\alpha^3 & \alpha ^4 \\
0 & \sqrt{d_1}z & A_1\sqrt{d_1}z & (A_1^2-B_1)\sqrt{d_1}z & (A_1^2-2B_1)\sqrt{d_1}z \\
0 & \sqrt{d_2}z & A_2\sqrt{d_2}z & (A_2^2-B_2)\sqrt{d_2}z & (A_2^2-2B_2)\sqrt{d_2}z \\
2 & A_1 & A_1^2-2B_1 & A_1(A_1^2-3B_1) & A_1^4-4A_1^2B_1+2B_1^2 \\
2 & A_2 & A_2^2-2B_2 & A_2(A_2^2-3B_2) & A_2^4-4A_2^2B_2+2B_2^2 \\
\end{pmatrix}.$$ Then, it is enough to determine $\sqrt{d_1}z$ and $\sqrt{d_2}z$.

For the first polynomial, we have $$\sqrt{d_1}=\frac{1}{6}\sqrt{-\frac{3}{65}}(3\alpha^4 + 15\alpha^3 - 25\alpha^2 - 110\alpha - 30).$$  Taking
$z=-\sqrt{-\dfrac {65}3}$ we get
$F=\mathbb{Q}_5(z)=\mathbb{Q}_5(\sqrt{5})$ and $$
\sqrt{d_1}z=\frac{1}{6}(3\alpha^4 + 15\alpha^3 - 25\alpha^2 - 110\alpha - 30).$$  
Then, we find that $$
\sqrt{d_2}z=
\frac{1}{6}(11\alpha^4 + 25\alpha^3 - 75\alpha^2 - 270\alpha - 30).
$$

For the second polynomial, one computes \begin{equation*}
    \sqrt{d_1}=\frac{1}{42}\frac{21}{\sqrt{235}}(\alpha^4+10\alpha^3+20\alpha^2+35\alpha-170).
\end{equation*} 
The element $z=\sqrt{235}$ satisfies $F=\mathbb{Q}_5(z)=\mathbb{Q}_5(\sqrt{10})$ and $$\sqrt{d_1}z=\frac{\alpha^4+10\alpha^3+20\alpha^2+35\alpha-170}{2}.$$ Under this consideration, one finds $$\sqrt{d_2}z=\frac{1}{42}(53\alpha^4+110\alpha^3+260\alpha^2-2195\alpha-190).$$

Finally, for the third polynomial, we have $$\sqrt{d_1}=\frac{1}{22}\sqrt{-\frac{1}{3}}(9\alpha^4+86\alpha^3+402\alpha^2+895\alpha+720).$$
Let $z=\sqrt{-3}$. Then $F=\mathbb{Q}_5(z)$ and $$\sqrt{d_1}z=\frac{1}{22}(9\alpha^4+86\alpha^3+402\alpha^2+895\alpha+720).$$
As for $d_2$, 
$$\sqrt{d_2}z=\frac{1}{22}(7\alpha^4+62\alpha^3+254\alpha^2+415\alpha+120).$$

\subsubsection{Basis of $\mathfrak{A}_{H_1}$}

We compute a basis for the associated order $\mathfrak A_{H_1}$. The matrices of the action for the first, second and third polynomial can be found in \eqref{matrixactionfirst} and \eqref{matrixactionthird}.

The Hermite normal form of $M(H_1,E)$ gives the matrix $D(H_1,E)$
$$
\left( \begin {array}{ccccc} 1&0&0&0&-1\\ 0&1&2&0&0\\ 0&0&5&0&0\\0&0&0&1&-1
\\ 0&0&0&0&5\end {array}
 \right),\ 
\left(\begin {array}{ccccc} 1&0&0&0&-1\\  0&1&-2&0&0
\\  0&0&5&0&0\\  0&0&0&1&-1
\\  0&0&0&0&5\end {array}
\right), \
\left(\begin {array}{ccccc} 1&0&0&0&-1\\  0&1&0&0&0
\\  0&0&1&0&0\\  0&0&0&1&-1
\\  0&0&0&0&5\end {array}
\right)
$$
respectively. 
The columns of the corresponding inverse provide a basis for the respecrive associated order:
$$
\left\lbrace w_1,w_2, \frac{-2w_2+w_3}5,w_4,\frac{w_1+w_4+w_5}{5}\right\rbrace
$$
$$\left\lbrace w_1,w_2,\frac{2w_2+w_3}{5},w_4,\frac{w_1+w_4+w_5}{5}\right\rbrace.$$

$$\left\lbrace w_1,w_2,w_3,w_4,\frac{w_1+w_4+w_5}{5}\right\rbrace.$$

\begin{rmk}\normalfont The element $\dfrac{w_1+w_4+w_5}{5}$ obtained in the three bases is actually $\dfrac{\sum_{i=0}^4\mu^i}{5}$. Likewise, for $p=3$, the element $\dfrac{w_1+w_3}{3}$ was obtained in all the bases in Section \ref{basisassocorder3part}, and it is $\dfrac{\mathrm{Id}+\mu+\mu^2}{3}$. For an arbitrary $p$, the element $\dfrac{\sum_{i=0}^{p-1}\mu^i}{p}$ (which acts on $E$ as $\dfrac{1}{p}$ times the trace map of $J=\langle\sigma\rangle$) always belongs to the associated order $\mathfrak{A}_{H_1}$. Indeed, its action on $\alpha$ gives $\dfrac{\sum_{i=1}^p\alpha_i}{p}$, where $\{\alpha_i\}_{i=1}^p$ are the conjugates of $\alpha$. Working with the symmetric functions of the roots, we see that this is $-1$ for the field defined by the polynomial $x^p+px^{p-1}+p$ and $0$ otherwise.
\end{rmk}

\subsubsection{Freeness over the associated order}

From the above we can compute the associated matrix $M_{\epsilon}(H_1,E)$ of an element
$\epsilon=\epsilon_1+\epsilon_2\alpha+\epsilon_3\alpha^2
+\epsilon_4\alpha^3 +\epsilon_5\alpha^4\in \mathcal O_E$, and its determinant allows us to determine whether $\epsilon$ is or not a free generator of $\mathcal{O}_L$ as $\mathfrak{A}_{H_1}$-module. We have:

\begin{center}
\begin{tabular}{ c ||c | c }
Case & $D_{\epsilon}(H_1,E)$  \\
\hline \hline
1  & $25 q_1(\epsilon_1,\epsilon_2,\epsilon_3,\epsilon_4,\epsilon_5)$ \\
\hline
2 & $50q_2(\epsilon_1,\epsilon_2,\epsilon_3,\epsilon_4,\epsilon_5)$ \\ \hline
3 & $10q_3(\epsilon_1,\epsilon_2,\epsilon_3,\epsilon_4,\epsilon_5)$ \\
\end{tabular}
\end{center}

For each $i\in\{1,2,3\}$, $q_i(\epsilon_1,\epsilon_2,\epsilon_3,\epsilon_4,\epsilon_5)$ is a product of homogeneous polynomials of degree at most $4$ (see \eqref{det5first}, \eqref{det5second} and \eqref{det5third} respectively). One may check that $q_i(1,1,1,1,1)$ is coprime with $5$ for $i\in\{1,2,3\}$. Hence, we obtain:

\begin{prop} If $E/\mathbb{Q}_5$ is a separable degree $5$ extension with dihedral degree $10$ Galois closure and $H_1$ is its unique Hopf Galois structure, then $\mathcal{O}_E$ is $\mathfrak{A}_{H_1}$-free.
\end{prop}




\section{Cyclic Hopf Galois module structure}\label{sechopfcyclic}

Let $L/\Qp$ be a dihedral degree $2p$ extension of $p$-adic fields. In this section we address the problem of the structure of $\mathcal{O}_L$ as module over the associated order in the Hopf Galois structures of cyclic type on $L/\mathbb{Q}_p$. As we have seen, those are the induced ones. As in the previous section, we first establish the strategy to study the action on $L$ and we solve completely the cases $p=3$ and $p=5$ afterwards.

\subsection{The general method}

Let us fix a degree $p$ subextension $E/\mathbb{Q}_p$ of $L/\mathbb{Q}_p$ and, as usual, call $F/\mathbb{Q}_p$ the quadratic subextension. We assume that $F/\Q_p$ is ramified, which implies that $L/\Q_p$ is totally ramified. We showed in Section \ref{subsectprimit} that the tensor basis $B=\{1,z,\alpha,\alpha z,\cdots,\alpha^{p-1},\alpha^{p-1}z\}$ is not a $\Zp$-basis of $\OL$ and the powers of the uniformizing parameter $\gamma=\dfrac{z}{\alpha^{\frac{p-1}{2}}}$ form an integral basis of $L$ $$B'=\{1,\gamma,\gamma^2,\dots,\gamma^{2p-1}\}.$$ Since the action of $H$ on $L$ is the tensor product of the actions of $H_1$ on $E$ and of $H_2$ on $F$, it is clear that $$G(H_W,L_B)=G(H_1,E)\otimes G(H_2,F).$$ Besides, we know from Section \ref{firstsubsection} that $G(H_W,L_{B'})=G(H_W,L_B)P_B^{B'}$. Then, in practice, carrying out the product of matrices above, one can compute the Gram matrix where in $L$ we fix an integral basis. However, powers of $\gamma$ greater than $2p-1$ normally appear in the result. In order to reduce them, we need the minimal polynomial of $\gamma$. To this end, we can use the theory of Tschirnhaus transformations, for example.

This yields the following method to compute $M(H_W,L_{B'})$:

\begin{itemize}
    \item[1.] Write the powers of $\gamma$ in the tensor basis $B$ to compute the matrix $P_B^{B'}$.
    \item[2.] Compute the Kronecker product $G(H_1,E)\otimes G(H_2,F)$ and multiply on left side by $P_B^{B'}$, obtaining $G(H_W,L_{B'})$.
    \item[3.] Compute the minimal polynomial of $\gamma$ and use it to find the coordinates of each entry of $G(H_W,L_{B'})$ with respect to $B'$.
    \item[4.] Compute $M(H_W,L_{B'})$ from the entries of $G(H_W,L_{B'})$.
\end{itemize}

Once we have computed $M(H_W,L_{B'})$, since $B'$ is an integral basis of $L$, we can apply the reduction method to compute a basis of $\mathfrak{A}_H$ and determine the $\mathfrak{A}_H$-freeness of $\OL$. As in the previous section, we work with the cases $p=3$ and $p=5$.

\subsection{The case $p=3$}

We can solve the singular case easily. Indeed, we have that $F/\Q_3$ is unramified. Since $\mathcal{O}_E$ is $\mathfrak{A}_{H_1}$-free, we can apply Corollary \ref{corounramified} to obtain that $\mathfrak{A}_H=\mathfrak{A}_{H_1}\otimes_{\Zp}\mathfrak{A}_{H_2}$ and $\mathcal{O}_L$ is $\mathfrak{A}_H$-free. Then, we deal with the totally ramified cases.

\subsubsection*{Change basis matrix}\label{sectchangebasisp=3}

Let $H=H_1\otimes H_2$ be an induced Hopf Galois structure of $L/\mathbb{Q}_p$. We fix as $\mathbb{Q}_3$-basis of $H$ the product of the bases in $H_1$ and $H_2$, that is, $$W\coloneqq\{w_1\eta_1,w_1\eta_2,w_2\eta_1,w_2\eta_2,w_3\eta_1,w_3\eta_2\}.$$

Since $L/\mathbb{Q}_3$ is totally ramified, by Proposition \ref{propunif}, the basis given by the powers of $\dfrac{z}{\alpha}$ is integral, where $z$ is the square root of a non-square in $\mathbb{Z}_3$ and $\alpha$ is a root of $f$. At each case, it is more convenient to choose $\gamma=t\frac{z}{\alpha}$ for a certain $t\in\mathbb{Z}_3^*$. Namely:
\begin{center}
\begin{tabular}{ c ||c | c | c }
Polynomial & $t$ & $z$ & $\gamma$  \\
\hline \hline
$x^3+3a,\,a\in\{1,4,7\}$ & $\sqrt{a}$ & $\sqrt{-3}$ & $-\dfrac{\alpha^2tz}{3a}$ \\ [1.5ex]
\hline
$x^3+3x+3$ & $\sqrt{\frac{1}{13}}$ & $\sqrt{-39}$ & $-\dfrac{(\alpha^2+3)tz}{3}$ \\ [1.5ex]
\hline
$x^3+6x+3$ & $\sqrt{-\frac{1}{41}}$ & $\sqrt{-123}$ & $-\dfrac{(\alpha^2+6)tz}{3}$ \\
\end{tabular}
\end{center}

Let $B'$ be the basis of the powers of $\gamma$. The coordinates of those powers with respect to the tensor basis $B$ determine the change basis matrix from the product basis $B$ to the basis $B'$. We obtain the following:

\begin{prop} The change basis matrix from $B$ to $B'$ is:
\begin{itemize}
    \item[1.] For the polynomials $f(x)=x^3+3a$, $a\in\{1,4,7\}$, $$P_B^{B'}= \left( \begin {array}{cccccc} 1&0&0&0&0&0\\  0&0&0&t
&0&0\\  0&0&1&0&0&0\\  0&0&0&0&0&t
\\  0&0&0&0&1&0\\  0&-\frac{1}{3t}&0
&0&0&0\end {array} 
\right).
$$
    \item[2.] For the polynomial $f(x)=x^3+3x+3$, $$P_B^{B'}=\left(\begin {array}{cccccc} 1&0&-3&0&15&0\\  0&-t&0
&4\,t&0&-18\,t\\  0&0&1&0&-3&0\\  0&0
&0&-t&0&4\,t\\  0&0&-1&0&4&0\\  0&-\frac{t}{3}&0&t&0&-5\,t\end {array}
\right).$$
    \item[3.] For the polynomial $f(x)=x^3+6x+3$, $$P_B^{B'}=\left(\begin {array}{cccccc} 1&0&12&0&156&0\\ \noalign{\medskip}0&-2
\,t&0&-25\,t&0&-324\,t\\ \noalign{\medskip}0&0&-1&0&-12&0
\\ \noalign{\medskip}0&0&0&2\,t&0&25\,t\\ \noalign{\medskip}0&0&2&0&25
&0\\ \noalign{\medskip}0&-\frac{t}{3}&0&-4\,t&0&-52\,t\end {array}
\right).$$
\end{itemize}
\end{prop}

\subsubsection*{The minimal polynomial of $\gamma$}

To compute the minimal polynomial of $\gamma$ for the radical cases
it is enough to remark that $$\gamma^6=(\gamma^3)^2=(tz)^2=-3a,$$ 
so $\gamma$ is a root of $Y^6+3a$. \\

For polynomials of the second group, we use the resultant. For $a=1$, we have $$\mathrm{Res}_x\left(x^3+3x+3,Y-\left(-\frac{(\alpha^2+3)tz}{3}\right)\right)=Y^3+tzY^2-tz=Y^3+(Y^2t-t)z.$$ This has root $\gamma$ as a polynomial in $Y$. Then, evaluating in $\gamma$ and squaring gives $$Y^6-(Y^2t-t)^2z^2=Y^6+3Y^4-6Y^2+3$$ the minimal polynomial of $\gamma$. For $a=2$, similarly we find the polynomial $$Y^6-12Y^4-12Y^2-3.$$

\subsubsection*{The action on $L/\mathbb{Q}_3$}

For the first three cases, $$G(H_1,E)\otimes G(H_2,F)= \left( \begin {array}{cccccc} 1&z&\alpha&\alpha\,z&{\alpha}^{2}&{
\alpha}^{2}z\\  1&-z&\alpha&-\alpha\,z&{\alpha}^{2}&-
{\alpha}^{2}z\\  0&0&-3\,\alpha&-3\,\alpha\,z&3\,{
\alpha}^{2}&3\,{\alpha}^{2}z\\  0&0&-3\,\alpha&3\,
\alpha\,z&3\,{\alpha}^{2}&-3\,{\alpha}^{2}z\\  2&2\,z
&-\alpha&-\alpha\,z&-{\alpha}^{2}&-{\alpha}^{2}z\\  2
&-2\,z&-\alpha&\alpha\,z&-{\alpha}^{2}&{\alpha}^{2}z\end {array}
 \right),
$$ whence we compute the Gram matrix
\begin{equation*}
    \begin{split}
        G(H_W,L_{B'})&=(G(H_1,E)\otimes G(H_2,F))P_B^{B'}\\&=\left(\begin {array}{cccccc} 1&\gamma&{\gamma}^{2}&{\gamma}^{3}&{
\gamma}^{4}&{\gamma}^{5}\\  1&-\gamma&{\gamma}^{2}&-{
\gamma}^{3}&{\gamma}^{4}&-{\gamma}^{5}\\  0&3\,\gamma
&-3\,{\gamma}^{2}&0&3\,{\gamma}^{4}&-3\,{\gamma}^{5}
\\  0&-3\,\gamma&-3\,{\gamma}^{2}&0&3\,{\gamma}^{4}&3
\,{\gamma}^{5}\\  2&-\gamma&-{\gamma}^{2}&2\,{\gamma}
^{3}&-{\gamma}^{4}&-{\gamma}^{5}\\  2&\gamma&-{\gamma
}^{2}&-2\,{\gamma}^{3}&-{\gamma}^{4}&{\gamma}^{5}\end {array}
\right),
    \end{split}
\end{equation*}
For the fourth and fifth polynomial the matrix $G(H_W,L_{B'})$ is obtained in a completely analogous way and their entries can be checked in \eqref{fourthgrammatrix} and \eqref{fifthgrammatrix} respectively.

\subsubsection*{Basis of $\mathfrak{A}_H$}

From the previous step we compute the matrix $M(H_W,L_{B'})$, and since $B'$ is an integral basis of $L$, reducing this matrix provides a basis of $\mathfrak{A}_H$.

For the radical cases, the Hermite normal form of $M(H_W,L_{B'})$ is $$D=\left( \begin {array}{cccccc} 1&0&0&0&-1&0\\  0&1&0&0
&0&-1\\  0&0&3&0&0&0\\  0&0&0&3&0&0
\\  0&0&0&0&3&0\\  0&0&0&0&0&3
\end {array} \right).
$$ Then, we obtain the basis of the associated order $\mathfrak{A}_H$ $$\left\lbrace w_1\eta_1,w_1\eta_2,\frac{w_2\eta_1}{3},\frac{w_2\eta_2}{3},\frac{w_1\eta_1+w_3\eta_1}{3},\frac{w_1\eta_2+w_3\eta_2}{3}\right\rbrace.$$ We see that $\mathfrak{A}_H=\mathfrak{A}_{H_1}\otimes_{\Z_3}\mathfrak{A}_{H_2}$.

For the second group of polynomials we get  $$D=\left(\begin {array}{cccccc} 1&0&0&0&0&-1\\  0&1&0&0
&0&-1\\  0&0&1&-1&0&0\\  0&0&0&3&0&0
\\  0&0&0&0&1&-1\\  0&0&0&0&0&3
\end {array} 
\right),$$ which gives that the associated order $\mathfrak{A}_H$ has $\mathbb{Z}_3$-basis $$\left\lbrace w_1\eta_1,w_1\eta_2,w_2\eta_1,\frac{w_2}{3}(\eta_1+\eta_2),w_3\eta_1,\frac{w_1+w_3}{3}(\eta_1+\eta_2)\right\rbrace.$$ In this case $\mathfrak{A}_H\neq\mathfrak{A}_{H_1}\otimes_{\mathbb{Z}_3}\mathfrak{A}_{H_2}$ since $\dfrac{w_2}{3}\eta_1\in\mathfrak{A}_{H_1}\otimes_{\mathbb{Z}_3}\mathfrak{A}_{H_2}$ and $\dfrac{w_2}{3}\eta_1\notin\mathfrak{A}_H$.

\subsubsection*{Freeness over the associated order}

Let $\beta=\sum_{i=1}^6\beta_i\gamma^{i-1}\in\mathcal{O}_L$. Then, we compute the determinant $D_{\beta}(H,L)$ of its associated matrix so as to determine whether it is a free generator of $\mathcal{O}_L$ as $\mathfrak{A}_H$-module. We obtain the following:

\begin{center}
\begin{tabular}{ l ||c | c | c }
Polynomial & $I(H,L)$ & $D_{\beta}(H,L)$  \\
\hline \hline
$x^3+3a$, $a\in\{1,4,7\}$  & $4$ & $-2592\beta_1\beta_2\beta_3\beta_4\beta_5\beta_6$  \\
\hline
$x^3+3x+3$ & $2$ & $-288\,q_2(\beta_1,\beta_2,\beta_3,\beta_4,\beta_5,\beta_6)$ \\ \hline
$x^3+6x+3$ & $2$ & $-288\,q_3(\beta_1,\beta_2,\beta_3,\beta_4,\beta_5,\beta_6)$ \\
\end{tabular}
\end{center}
Here $q_2$ and $q_3$ are the products of linear and quadratic forms:
\begin{equation*}
    \begin{split}
        &q_2(\beta_1,\beta_2,\beta_3,\beta_4,\beta_5,\beta_6)=\, \left( 3\,{{\beta_3}}^{2}-23\,{\beta_3}\,{\beta_5}+43\,{{\beta_5}}^{2} \right)  \left( {\beta_2}-6\,{\beta_4}+24\,{\beta_6} \right) \\& \left( {{\beta_2}}^{2}-15\,{\beta_2}\,{\beta_4}+66\,{\beta_6}\,{\beta_2}+27\,{{\beta_4}}^{2}-261\,{\beta_6}\,{\beta_4}+621\,{{\beta_6}}^{2} \right) \left( {\beta_1}-{\beta_3}+7\,{\beta_5} \right), \\
        &q_3(\beta_1,\beta_2,\beta_3,\beta_4,\beta_5,\beta_6)=\,\left(20\,{\beta_3}^{2}+499\,\beta_3\,\beta_5+3112\,{\beta_5}^{2} \right) \left( 2\,\beta_2+27\,\beta_4+348\,\beta_6 \right)\,\\&\Big( 4\,{\beta_2}^{2}+114\,\beta_2\,\beta_4 +1473\,\beta_6\,\beta_2+720\,{\beta_4}^{2}+18684\,\beta_6\,\beta_4+121194\,{\beta_6}^{2} \Big)\\&  \left( \beta_1+4\,\beta_3+56\,\beta_5 \right).
    \end{split}
\end{equation*}

Since $v_3(2592)=4$ and $v_3(288)=2$, we can always find $\beta\in\mathcal{O}_L$ such that $I(H,L)=D_{\beta}(H,L)$: take $\beta=\sum_{i=1}^6\gamma^{i-1}$ for the first case and $\beta=\gamma+\gamma^4$ in the other two. Then:

\begin{prop} If $L/\mathbb{Q}_3$ is a dihedral degree $6$ extension of $3$-adic fields and $H$ is a Hopf Galois structure of cyclic type on $L/\mathbb{Q}_3$, then $\mathcal{O}_L$ is $\mathfrak{A}_H$-free.
\end{prop}

\subsubsection*{The product of generators}

Let us check that the product $\beta'$ of the generators $\epsilon$ of $\mathcal{O}_E$ as $\mathfrak{A}_{H_1}$-module and $\delta$ of $\mathcal{O}_F$ as $\mathfrak{A}_{H_2}$-module is not a generator of $\mathcal{O}_L$ as $\mathfrak{A}_H$-module.

Such a product is of the form $$\beta'=(\epsilon_1+\epsilon_2\alpha+\epsilon_3\alpha^2)(\delta_1+\delta_2z)=\epsilon_1\delta_1+\epsilon_1\delta_2z+\epsilon_2\delta_1\alpha+\epsilon_2\delta_2\alpha z+\epsilon_3\delta_1\alpha^2+\epsilon_3\delta_2\alpha^2z$$ with $\epsilon_i,\delta_j\in\mathbb{Z}_3$ such that $\epsilon_1\epsilon_2\epsilon_3,\,\delta_1\delta_2\in\mathbb{Z}_3^*$.
Note that this refers to the tensor basis $B$. Then, we apply the matrix $P_{B'}^B$ so as to obtain the coordinates of $\beta'$ with respect to the basis $B'$.
\begin{itemize}
    \item For the radical cases, $$D_{\beta'}(H,L)=7776\,{\frac {{\epsilon_2}^{2}{\delta_2}^{3}{\epsilon_3}^{2}{\delta_1}^{3}
{\epsilon_1}^{2}}{t}}.$$
    \item If $f(x)=x^3+3x+3$, $$D_{\beta'}(H,L)=7776\,{\frac {{\delta_1}^{3} \left( {\epsilon_2}^{2}+3\,\epsilon_2\,\epsilon_3-{\epsilon_3}^{2} \right) ^{2}{\delta_2}^{3} \left( \epsilon_1-2\,\epsilon_3 \right) ^{2}}{{t}^{3}}}.
$$
    \item If $f(x)=x^3+6x+3$, $$D_{\beta'}(H,L)=31104\,{\frac { {{\delta_1}}^{3}\left( {{\epsilon_2}}^{2}+\frac{3}{2}\,{\epsilon_2}\,{\epsilon_3}-2\,{{\epsilon_3}}^{2} \right) ^{2}{{\delta_2}}^{3} \left( {\epsilon_1}-4\,{\epsilon_3} \right) ^{2}}{{t}^{3}}}$$
\end{itemize}

We have that $v_3(D_{\beta'}(H,L)\geq5$ in all cases, so $\beta'$ is not a free generator.

\subsection{The case $p=5$}

Now, we consider a dihedral degree $10$ extension $L/\mathbb{Q}_5$ of $5$-adic fields. Recall that $L$ is the splitting field over $\mathbb{Q}_5$ of one of the polynomials $$x^5+15x^2+5,\quad x^5+10x^2+5,\quad x^5+5x^4+5.$$ We have seen in Section \ref{secdihedral} that the unique case in which $F/\mathbb{Q}_5$ is unramified is the third one. As in the case $p=3$, we have that $\mathfrak{A}_H=\mathfrak{A}_{H_1}\otimes_{\Zp}\mathfrak{A}_{H_2}$ and $\mathcal{O}_L$ is $\mathfrak{A}_H$-free for every induced Hopf Galois structure $H=H_1\otimes_{\Qp}H_2$. 

Hence, throughout this section we consider the first two cases. Recall that we have replaced the first two polynomials by $$x^5-15x^3-10x^2+75x+30,\quad x^5-35x^2+50x+20,$$ respectively.

\subsubsection*{Integral basis of $L$ and minimal polynomial}\label{sectpowersgamma5}

By Proposition \ref{propunif}, the powers of $\dfrac{z}{\alpha^2}$ form an integral basis $B'$ of $L$. Again, we adjust this element by certain $t\in\mathbb{Z}_5^*$. Concretely:

\begin{center}
\begin{tabular}{ c ||c | c | c }
Case & $t$ & $z$ & $\gamma$  \\
\hline \hline
1 & $-\sqrt{-\dfrac{3}{13}}$ & $-\sqrt{-\dfrac {65}3}$ & $\dfrac{1}{5}\, \left( 5\,{\alpha}^{4}-2\,{\alpha}^{3}-75\,{\alpha}^{2}-20\,
\alpha+395 \right) tz$ \\[1.5ex]
\hline
2 & $-\sqrt{-\dfrac{2}{47}}$ & $\sqrt{235}$ & $\dfrac{1}{5}\left( 5\,{\alpha}^{4}-2\,{\alpha}^{3}-175\,\alpha+320 \right) t
z$ \\
\end{tabular}
\end{center}

Proceeding as before, we compute the minimal polynomial of $\gamma$ at each case, which are
$${Y}^{10}-30685\,{Y}^{8}+580960\,{Y}^{6}-5564160\,{Y}^{4}+49766400\,{Y}
^{2}-238878720
$$ and
 $${Y}^{10}-56920\,{Y}^{8}+1844800\,{Y}^{6}-29163520\,{Y}^{4}-209715200\,
{Y}^{2}-671088640.$$
respectively.

\subsubsection*{Basis of $\mathfrak{A}_H$}



For the first polynomial, the Hermite normal form of $M(H_W,L_{B'})$ is $$D=\left(\begin {array}{cccccccccc} 1&0&0&0&0&0&0&0&0&-1
\\  0&1&0&0&0&0&0&0&0&-1\\  0&0&1&0&0
&2&0&0&0&0\\  0&0&0&1&0&2&0&0&0&0
\\  0&0&0&0&1&-1&0&0&0&0\\  0&0&0&0&0
&5&0&0&0&0\\  0&0&0&0&0&0&1&0&0&-1
\\  0&0&0&0&0&0&0&1&0&-1\\  0&0&0&0&0
&0&0&0&1&-1\\  0&0&0&0&0&0&0&0&0&5\end {array}
\right),$$ which gives the basis of $\mathfrak{A}_H$ \begin{equation*}
   \begin{split}\Big\{
        &w_1\eta_1,w_1\eta_2,w_2\eta_1,w_2\eta_2,w_3\eta_1,\frac{-2w_2+w_3}{5}(\eta_1+\eta_2),\\&w_4\eta_1,w_4\eta_2,w_5\eta_1,\frac{w_1+w_4+w_5}{5}(\eta_1+\eta_2)\Big\}.
    \end{split}
\end{equation*}

For the second one, we obtain as Hermite normal form $$D=\left(\begin {array}{cccccccccc} 1&0&0&0&0&0&0&0&0&-1
\\  0&1&0&0&0&0&0&0&0&-1\\  0&0&1&0&0
&-2&0&0&0&0\\  0&0&0&1&0&-2&0&0&0&0
\\  0&0&0&0&1&-1&0&0&0&0\\  0&0&0&0&0
&5&0&0&0&0\\  0&0&0&0&0&0&1&0&0&-1
\\  0&0&0&0&0&0&0&1&0&-1\\  0&0&0&0&0
&0&0&0&1&-1\\  0&0&0&0&0&0&0&0&0&5\end {array}
\right),$$ giving the basis of $\mathfrak{A}_H$ 
\begin{equation*}
   \begin{split}\Big\{
        &w_1\eta_1,w_1\eta_2,w_2\eta_1,w_2\eta_2,w_3\eta_1,\frac{2w_2+w_3}{5}(\eta_1+\eta_2),\\&w_4\eta_1,w_4\eta_2,w_5\eta_1,\frac{w_1+w_4+w_5}{5}(\eta_1+\eta_2)\Big\}.
    \end{split}
\end{equation*} 

In both cases, $\mathfrak{A}_H\neq\mathfrak{A}_{H_1}\otimes_{\mathbb{Z}_5}\mathfrak{A}_{H_2}$.

\subsubsection*{Freeness over $\mathfrak{A}_H$}

Let $\beta=\sum_{i=1}^6\beta_i\gamma^{i-1}\in\mathcal{O}_L$. Using the matrix $M(H_W,L_B)$, 
we can determine the matrix $M_{\beta}(H_W,L_B)$ associated to $\beta$, whose determinant allows us to determine whether or not $\mathcal{O}_L$ is $\mathfrak{A}_H$-free. In both cases, we have $I(H,L)=2$, and if $$\beta=1+\gamma+\gamma^2+\gamma^3+\gamma^4+\gamma^5+\gamma^6+\gamma^7+\gamma^8+\gamma^9,$$ $v_5(D_{\beta}(H,L))=2$, so $\mathcal{O}_L$ is $\mathfrak{A}_H$-free with generator $\beta$.

\subsubsection*{The product of generators}

Let us check if the product of generators for $E$ and $F$ is a generator for $L$. Let $\epsilon=\sum_{i=1}^5\epsilon_i\alpha^{i-1}$ be an $\mathfrak{A}_{H_1}$-generator of $\mathcal{O}_E$ and $\delta=\delta_1+\delta_2z$ an $\mathfrak{A}_{H_2}$-generator of $\mathcal{O}_F$. That product is \begin{equation*}
    \begin{split}
        \beta'=\epsilon_1\delta_1+\epsilon_1\delta_2z+\epsilon_2\delta_1\alpha+\epsilon_2\delta_2\alpha z+\epsilon_3\delta_1\alpha^2+\epsilon_3\delta_2\alpha^2z+\epsilon_4\delta_1\alpha^3+\epsilon_4\delta_2\alpha^3z\\+\epsilon_5\delta_1\alpha^4+\epsilon_5\delta_2\alpha^4z,
    \end{split}
\end{equation*} and applying $P_{B'}^B$ on the column of its coordinates, we obtain its vector of coordinates $(\beta_i')_{i=1}^{10}$ with respect to $B'$. The vectors for each case are shown in \eqref{coordproductfirstpolyn5} and \eqref{coordproductsecondpolyn5}. If we set these coordinates to the previously computed determinant $D_{\beta}(H,L)$, we find that $v_5(D_{\beta'}(H,L))>2$, so $\beta'$ is not a $\mathfrak{A}_H$-generator of $\mathcal{O}_L$. 

\subsection{Summary of results}

We summarize the results we have obtained for the case $p=3$. With the notation used throughout the corresponding section:

\begin{thm}[Associated orders]\label{assocorderdih6}
\begin{itemize}
    \item[1.] For the first three polynomials and the last one, $\mathfrak{A}_H=\mathfrak{A}_{H_1}\otimes_{\mathbb{Z}_3}\mathfrak{A}_{H_2}$.
    \item[2.] For the fourth and the fifth polynomials $\mathfrak{A}_H\neq\mathfrak{A}_{H_1}\otimes_{\mathbb{Z}_3}\mathfrak{A}_{H_2}$
    and a basis of $\mathfrak{A}_H$ is 
    $$\left\lbrace w_1\eta_1,w_1\eta_2,w_2\eta_1,\frac{w_2}{3}(\eta_1+\eta_2),w_3\eta_1,\frac{w_1+w_3}{3}(\eta_1+\eta_2)\right\rbrace.$$
\end{itemize}
\end{thm}

\begin{thm}[Freeness]\label{teofreenessdih6} $\mathcal{O}_L$ is $\mathfrak{A}_H$-free in all cases. For the last polynomial the product of a generator of $\mathcal{O}_E$ as $\mathfrak{A}_{H_1}$-module and a generator of $\mathcal{O}_F$ as $\mathfrak{A}_{H_2}$-module is a generator of $\mathcal{O}_L$ as $\mathfrak{A}_H$-module, while in the rest of the cases such a product is never a generator.
\end{thm}

For $p=5$, we obtain the following results:

\begin{thm}[Associated orders]
\begin{itemize}
    \item[1.] The equality $\mathfrak{A}_H=\mathfrak{A}_{H_1}\otimes_{\mathbb{Z}_5}\mathfrak{A}_{H_2}$ holds only for the third polynomial.
    \item[2.] For the first polynomial, a basis of $\mathfrak{A}_H$ is \begin{equation*}
   \begin{split}\Big\{
        &w_1\eta_1,w_1\eta_2,w_2\eta_1,w_2\eta_2,w_3\eta_1,\frac{-2w_2+w_3}{5}(\eta_1+\eta_2),\\&w_4\eta_1,w_4\eta_2,w_5\eta_1,\frac{w_1+w_4+w_5}{5}(\eta_1+\eta_2)\Big\},
    \end{split}
\end{equation*} while for the second polynomial, a basis is
    \begin{equation*}
    \begin{split}\Big\{
        &w_1\eta_1,w_1\eta_2,w_2\eta_1,w_2\eta_2,w_3\eta_1,\frac{2w_2+w_3}{5}(\eta_1+\eta_2),\\&w_4\eta_1,w_4\eta_2,w_5\eta_1,\frac{w_1+w_4+w_5}{5}(\eta_1+\eta_2)\Big\}.
    \end{split}
\end{equation*}
\end{itemize}
\end{thm}

\begin{thm}[Freeness]\label{teofreenessdih10} $\mathcal{O}_L$ is $\mathfrak{A}_H$-free for all cases. Only for the last polynomial the product of a generator of $\mathcal{O}_E$ as $\mathfrak{A}_{H_1}$-module and a generator of $\mathcal{O}_F$ as $\mathfrak{A}_{H_2}$-module is a generator of $\mathcal{O}_L$ as $\mathfrak{A}_H$-module.
\end{thm}

\section{Conclusions}

Our results refer to the Hopf Galois module structure of not only dihedral degree $2p$ extensions of $\mathbb{Q}_p$, but also their degree $p$ subextensions, that is, separable degree $p$ extensions of $\mathbb{Q}_p$ with dihedral degree $2p$ Galois closure. These extensions are interesting by themselves because they are Hopf Galois non-Galois extensions. Therefore, questions about the Galois module structure of the ring of integers refer necessarily to its Hopf Galois structure. 

A precedent in the study of such extensions can be found in the paper \cite{elder} of Elder, where he deals with what he calls typical degree $p$ extensions: totally ramified degree $p$ extensions $L/K$ of local fields that \textbf{are not} generated by the $p$-th root of an uniformiser of $K$. Our degree $p$ extensions $E/\mathbb{Q}_p$ are typical (except the radical ones), but are not covered by his result \cite[Corollary 3.6]{elder}.

However, that result shows a strong connection between the ramification of the extension and freeness of $\mathcal{O}_L$, which is coherent with our results. For $p\in\{3,5\}$, the two cases $f(x)=x^p+apx^{\frac{p-1}{2}}+p$, $a\in\{2,p-2\}$, have the same ramification invariants, and they present a very similar behaviour: for instance, $I(H_1,E)=2$ for both cases. However, for the singular case, we have $I(H_1,E)=1$. As for the radical cases in $p=3$, we obtained the same associated order as polynomials of the second group, but the determinant $D_{\beta}(H_1,E)$ of the matrix associated to an element $\beta$ was substantially different.

Going back to the dihedral degree $2p$ extensions $L/\mathbb{Q}_p$, the idea behind most of the sections is to make use of induced Hopf Galois structures to recover properties from every pair of a degree $p$ and a quadratic subextension. In the case that these two extensions are arithmetically disjoint (which corresponds to $L/\mathbb{Q}_p$ not being totally ramified), we know that $\mathfrak{A}_H=\mathfrak{A}_{H_1}\otimes_{\mathcal{O}_K}\mathfrak{A}_{H_2}$ and that the freeness of $\mathcal{O}_L$ is implied by the analog property in the subextensions (we could say that induction behaves well with the associated order and the freeness). 

Actually, our results suggest a connection between these concepts stronger than arithmetic disjointness. For instance, for the radical cases, we have that $\mathfrak{A}_H=\mathfrak{A}_{H_1}\otimes_{\mathcal{O}_K}\mathfrak{A}_{H_2}$ even though there is not arithmetic disjointness. However, the product of generators of $\mathcal{O}_E$ and $\mathcal{O}_F$ over their corresponding associated orders is not still a generator of $\mathcal{O}_L$. As for the polynomials $f(x)=x^p+apx^{\frac{p-1}{2}}+p$, $a\in\{2,p-2\}$, $\mathfrak{A}_H\neq\mathfrak{A}_{H_1}\otimes_{\mathbb{Z}_p}\mathfrak{A}_{H_2}$, but the likeness between the bases of $\mathfrak{A}_{H_1}$ and $\mathfrak{A}_{H_2}$ obtained for $p\in\{3,5\}$ suggest another kind of relation.

In light of the results obtained, for a general $p$, it is reasonable to expect freeness over the associated order when the ground field is $\mathbb{Q}_p$, or at least, when $p>3$, the same behaviour for the two totally ramified cases. As for the radical cases, the Galois closure of a radical degree $p$ polynomial is a Frobenius degree $p(p-1)$. Then, the reason why radical polynomials have only appeared for $p=3$ if that it is the unique prime number for which $2p=p(p-1)$. Thus, it seems more likely to expect similar properties for the class of Frobenius extensions of $\mathbb{Q}_p$.

\newpage
\appendix

\section{Complete form of some objects}\label{appendix}

\subsection*{Matrices of the action}

\subsubsection*{$f(x)=x^5+15x^2+5$\qquad\qquad\qquad   $f(x)=x^5+10x^2+5$}

\begin{equation}\label{matrixactionfirst}
\left( \begin {array}{ccccc} 6&0&0&12&12\\ \noalign{\medskip}0&0&0&0&0
\\ \noalign{\medskip}0&0&0&0&0\\ \noalign{\medskip}0&0&0&0&0
\\ \noalign{\medskip}0&0&0&0&0\\ \noalign{\medskip}0&-30&-30&30&-30
\\ \noalign{\medskip}6&-110&-270&-18&12\\ \noalign{\medskip}0&-25&-75&
-11&11\\ \noalign{\medskip}0&15&25&1&-1\\ \noalign{\medskip}0&3&11&1&-
1\\ \noalign{\medskip}0&120&-360&120&60\\ \noalign{\medskip}0&205&-15&
45&-45\\ \noalign{\medskip}6&55&55&7&-13\\ \noalign{\medskip}0&-25&5&-
5&5\\ \noalign{\medskip}0&-10&0&-2&2\\ \noalign{\medskip}0&150&750&210
&-30\\ \noalign{\medskip}0&-245&-1815&-285&285\\ \noalign{\medskip}0&-
160&-700&-110&110\\ \noalign{\medskip}6&60&150&22&-28
\\ \noalign{\medskip}0&15&85&13&-13\\ \noalign{\medskip}0&450&-5250&
1110&-210\\ \noalign{\medskip}0&-1100&-2250&600&-600
\\ \noalign{\medskip}0&-135&255&55&-55\\ \noalign{\medskip}0&85&295&-
65&65\\ \noalign{\medskip}6&-5&65&-23&17\end {array}
 \right),    
\
\left( \begin {array}{ccccc} 42&0&0&84&84\\ \noalign{\medskip}0&0&0&0
&0\\ \noalign{\medskip}0&0&0&0&0\\ \noalign{\medskip}0&0&0&0&0
\\ \noalign{\medskip}0&0&0&0&0\\ \noalign{\medskip}0&-3570&-190&-170&
170\\ \noalign{\medskip}42&735&-2195&-1&-41\\ \noalign{\medskip}0&420&
260&16&-16\\ \noalign{\medskip}0&210&110&10&-10\\ \noalign{\medskip}0&
21&53&1&-1\\ \noalign{\medskip}0&10160&2180&-160&160
\\ \noalign{\medskip}0&-14975&-275&655&-655\\ \noalign{\medskip}42&830
&1040&-106&64\\ \noalign{\medskip}0&440&20&-40&40\\ \noalign{\medskip}0
&485&65&-25&25\\ \noalign{\medskip}0&-3000&13850&3380&1030
\\ \noalign{\medskip}0&29790&17320&-830&830\\ \noalign{\medskip}0&-
3930&-6010&50&-50\\ \noalign{\medskip}42&-1800&-2050&-16&-26
\\ \noalign{\medskip}0&-1020&-730&32&-32\\ \noalign{\medskip}0&-128150
&-15650&-9990&1590\\ \noalign{\medskip}0&34625&-109675&-375&375
\\ \noalign{\medskip}0&15520&15160&540&-540\\ \noalign{\medskip}0&6550
&6850&390&-390\\ \noalign{\medskip}42&235&3205&39&-81\end {array}
\right)
\end{equation}

are the matrices $6M(H_1,E)$ and $42M(H_1,E)$, respectively.
\newpage

\subsubsection*{$f(x)=x^5+5x^4+5$}

\begin{equation}\label{matrixactionthird}
M(H_1,E)=\frac{1}{22} \left(\begin {array}{ccccc} 22&0&0&44&44\\ \noalign{\medskip}0&0&0&0
&0\\ \noalign{\medskip}0&0&0&0&0\\ \noalign{\medskip}0&0&0&0&0
\\ \noalign{\medskip}0&0&0&0&0\\ \noalign{\medskip}0&720&120&-40&-180
\\ \noalign{\medskip}22&895&415&199&-221\\ \noalign{\medskip}0&402&254
&150&-150\\ \noalign{\medskip}0&86&62&38&-38\\ \noalign{\medskip}0&9&7
&5&-5\\ \noalign{\medskip}0&-4410&-330&-830&830\\ \noalign{\medskip}0&
-6025&-1155&-1555&1555\\ \noalign{\medskip}22&-2960&-550&-996&974
\\ \noalign{\medskip}0&-656&-132&-240&240\\ \noalign{\medskip}0&-73&-
11&-31&31\\ \noalign{\medskip}0&11520&-2010&6020&-3270
\\ \noalign{\medskip}0&17950&-3500&6510&-6510\\ \noalign{\medskip}0&
9930&-2830&3550&-3550\\ \noalign{\medskip}22&2300&-670&804&-826
\\ \noalign{\medskip}0&276&-98&100&-100\\ \noalign{\medskip}0&16620&
23040&-19700&6500\\ \noalign{\medskip}0&10535&46845&-11525&11525
\\ \noalign{\medskip}0&-1550&28550&-3730&3730\\ \noalign{\medskip}0&-
950&6690&-590&590\\ \noalign{\medskip}22&-235&805&-51&29\end {array}
 \right)
\end{equation}

\subsection*{Non-constant factors of $D_{\epsilon}(H_1,E)$}

\subsubsection*{$f(x)=x^5+15x^2+5$}

\begin{equation}\label{det5first}
    \begin{split}
        &q_1(\epsilon_1,\epsilon_2,\epsilon_3,\epsilon_4,\epsilon_5)=-\frac{1}{27}\, \Big( 11\,{{\epsilon_2}}^{4}+35\,{{\epsilon_2}}^{3}{\epsilon_3}+215\,{{\epsilon_2}}^{3}{\epsilon_4}+830\,{{\epsilon_2}}^{3}{\epsilon_5}-75\,{{\epsilon_2}}^{2}{{\epsilon_3}}^{2}\\&+795\,{{\epsilon_2}}^{2}{\epsilon_3}\,{\epsilon_4}-465\,{{\epsilon_2}}^{2}{\epsilon_3}\,{\epsilon_5}+1080\,{{\epsilon_2}}^{2}{{\epsilon_4}}^{2}+16485\,{{\epsilon_2}}^{2}{\epsilon_4}\,{\epsilon_5}+9000\,{{\epsilon_2}}^{2}{{\epsilon_5}}^{2}\\&+5\,{\epsilon_2}\,{{\epsilon_3}}^{3}-615\,{\epsilon_2}\,{{\epsilon_3}}^{2}{\epsilon_4}-1650\,{\epsilon_2}\,{{\epsilon_3}}^{2}{\epsilon_5}+4080\,{\epsilon_2}\,{\epsilon_3}\,{{\epsilon_4}}^{2}+8550\,{\epsilon_2}\,{\epsilon_3}\,{\epsilon_4}\,{\epsilon_5}\\&-30075\,{\epsilon_2}\,{\epsilon_3}\,{{\epsilon_5}}^{2}+200\,{\epsilon_2}\,{{\epsilon_4}}^{3}+79650\,{\epsilon_2}\,{{\epsilon_4}}^{2}{\epsilon_5}+240525\,{\epsilon_2}\,{\epsilon_4}\,{{\epsilon_5}}^{2}-112750\,{\epsilon_2}\,{{\epsilon_5}}^{3}\\&+5\,{{\epsilon_3}}^{4}-205\,{{\epsilon_3}}^{3}{\epsilon_4}+725\,{{\epsilon_3}}^{3}{\epsilon_5}+1680\,{{\epsilon_3}}^{2}{{\epsilon_4}}^{2}-24975\,{{\epsilon_3}}^{2}{\epsilon_4}\,{\epsilon_5}+18825\,{{\epsilon_3}}^{2}{{\epsilon_5}}^{2}\\&-3430\,{\epsilon_3}\,{{\epsilon_4}}^{3}+135450\,{\epsilon_3}\,{{\epsilon_4}}^{2}{\epsilon_5}-373725\,{\epsilon_3}\,{\epsilon_4}\,{{\epsilon_5}}^{2}+148625\,{\epsilon_3}\,{{\epsilon_5}}^{3}-1045\,{{\epsilon_4}}^{4}\\&-16600\,{{\epsilon_4}}^{3}{\epsilon_5}+1473450\,{{\epsilon_4}}^{2}{{\epsilon_5}}^{2}-1173625\,{\epsilon_4}\,{{\epsilon_5}}^{3}+210125\,{{\epsilon_5}}^{4} \Big)  \left( {\epsilon_1}+6\,{\epsilon_3}+6\,{\epsilon_4}+30\,{\epsilon_5} \right) .
    \end{split}
\end{equation}

\subsubsection*{$f(x)=x^5+10x^2+5$}

\begin{equation}\label{det5second}
    \begin{split}
        &q_2(\epsilon_1,\epsilon_2,\epsilon_3,\epsilon_4,\epsilon_5)=-\frac{1}{21}\Big( 50\,{{\epsilon_2}}^{4}-750\,{{\epsilon_2}}^{3}{\epsilon_3}-500\,{{\epsilon_2}}^{3}{\epsilon_4}+13000\,{{\epsilon_2}}^{3}{\epsilon_5}\\&-24250\,{{\epsilon_2}}^{2}{{\epsilon_3}}^{2}+99250\,{{\epsilon_2}}^{2}{\epsilon_3}\,{\epsilon_4}-61250\,{{\epsilon_2}}^{2}{\epsilon_3}\,{\epsilon_5}-77500\,{{\epsilon_2}}^{2}{{\epsilon_4}}^{2}-218000\,{{\epsilon_2}}^{2}{\epsilon_4}\,{\epsilon_5}\\&+1137500\,{{\epsilon_2}}^{2}{{\epsilon_5}}^{2}+48750\,{\epsilon_2}\,{{\epsilon_3}}^{3}+21250\,{\epsilon_2}\,{{\epsilon_3}}^{2}{\epsilon_4}-1792500\,{\epsilon_2}\,{{\epsilon_3}}^{2}{\epsilon_5}-597750\,{\epsilon_2}\,{\epsilon_3}\,{{\epsilon_4}}^{2}\\&+7095000\,{\epsilon_2}\,{\epsilon_3}\,{\epsilon_4}\,{\epsilon_5}-1711250\,{\epsilon_2}\,{\epsilon_3}\,{{\epsilon_5}}^{2}+628750\,{\epsilon_2}\,{{\epsilon_4}}^{3}-4655000\,{\epsilon_2}\,{{\epsilon_4}}^{2}{\epsilon_5}\\&-14302500\,{\epsilon_2}\,{\epsilon_4}\,{{\epsilon_5}}^{2}+40375000\,{\epsilon_2}\,{{\epsilon_5}}^{3}-23750\,{{\epsilon_3}}^{4}-97500\,{{\epsilon_3}}^{3}{\epsilon_4}+1756250\,{{\epsilon_3}}^{3}{\epsilon_5}\\&+900000\,{{\epsilon_3}}^{2}{{\epsilon_4}}^{2}-908750\,{{\epsilon_3}}^{2}{\epsilon_4}\,{\epsilon_5}-31756250\,{{\epsilon_3}}^{2}{{\epsilon_5}}^{2}-1651250\,{\epsilon_3}\,{{\epsilon_4}}^{3}\\&-15948750\,{\epsilon_3}\,{{\epsilon_4}}^{2}{\epsilon_5}+125468750\,{\epsilon_3}\,{\epsilon_4}\,{{\epsilon_5}}^{2}-18518750\,{\epsilon_3}\,{{\epsilon_5}}^{3}+907250\,{{\epsilon_4}}^{4}\\&+18321250\,{{\epsilon_4}}^{3}{\epsilon_5}-73877500\,{{\epsilon_4}}^{2}{{\epsilon_5}}^{2}-240725000\,{\epsilon_4}\,{{\epsilon_5}}^{3}+495931250\,{{\epsilon_5}}^{4} \Big)\\&  \left( {\epsilon_1}+21\,{\epsilon_4}-40\,{\epsilon_5} \right) .
    \end{split}
\end{equation}

\subsubsection*{$f(x)=x^5+5x^4+5$}

\begin{equation}\label{det5third}
    \begin{split}
        &q_3(\epsilon_1,\epsilon_2,\epsilon_3,\epsilon_4,\epsilon_5)=\frac{1}{11}({\epsilon_1}-2\,{\epsilon_2}+25\,{\epsilon_4}-120\,{\epsilon_5})\Big({{\epsilon_2}}^{4}-25\,{{\epsilon_2}}^{3}{\epsilon_3}+70\,{{\epsilon_2}}^{3}{\epsilon_4}+50\,{{\epsilon_2}}^{3}{\epsilon_5}\\&+215\,{{\epsilon_2}}^{2}{{\epsilon_3}}^{2}-1035\,{{\epsilon_2}}^{2}{\epsilon_3}\,{\epsilon_4}-1895\,{{\epsilon_2}}^{2}{\epsilon_3}\,{\epsilon_5}+810\,{{\epsilon_2}}^{2}{{\epsilon_4}}^{2}+10050\,{{\epsilon_2}}^{2}{\epsilon_4}\,{\epsilon_5}-13300\,{{\epsilon_2}}^{2}{{\epsilon_5}}^{2}\\&-755\,{\epsilon_2}\,{{\epsilon_3}}^{3}+4585\,{\epsilon_2}\,{{\epsilon_3}}^{2}{\epsilon_4}+15000\,{\epsilon_2}\,{{\epsilon_3}}^{2}{\epsilon_5}-4725\,{\epsilon_2}\,{\epsilon_3}\,{{\epsilon_4}}^{2}-115900\,{\epsilon_2}\,{\epsilon_3}\,{\epsilon_4}\,{\epsilon_5}\\&+70725\,{\epsilon_2}\,{\epsilon_3}\,{{\epsilon_5}}^{2}-3275\,{\epsilon_2}\,{{\epsilon_4}}^{3}+144750\,{\epsilon_2}\,{{\epsilon_4}}^{2}{\epsilon_5}+230050\,{\epsilon_2}\,{\epsilon_4}\,{{\epsilon_5}}^{2}-949250\,{\epsilon_2}\,{{\epsilon_5}}^{3}\\&+895\,{{\epsilon_3}}^{4}-5600\,{{\epsilon_3}}^{3}{\epsilon_4}-32525\,{{\epsilon_3}}^{3}{\epsilon_5}+650\,{{\epsilon_3}}^{2}{{\epsilon_4}}^{2}+293025\,{{\epsilon_3}}^{2}{\epsilon_4}\,{\epsilon_5}+38925\,{{\epsilon_3}}^{2}{{\epsilon_5}}^{2}\\&+30575\,{\epsilon_3}\,{{\epsilon_4}}^{3}-496025\,{\epsilon_3}\,{{\epsilon_4}}^{2}{\epsilon_5}-2641125\,{\epsilon_3}\,{\epsilon_4}\,{{\epsilon_5}}^{2}+4970875\,{\epsilon_3}\,{{\epsilon_5}}^{3}-20225\,{{\epsilon_4}}^{4}\\&-212875\,{{\epsilon_4}}^{3}{\epsilon_5}+6795500\,{{\epsilon_4}}^{2}{{\epsilon_5}}^{2}-10711250\,{\epsilon_4}\,{{\epsilon_5}}^{3}-8010125\,{{\epsilon_5}}^{4}\Big).
    \end{split}
\end{equation}

\subsection*{Gram matrices}

\subsubsection*{$f(x)=x^3+3x+3$}

\begin{equation}\label{fourthgrammatrix}
    G(H_W,L_B)=\begin{pmatrix}
    1 & \gamma & \gamma^2 & \gamma^3 & \gamma^4 & \gamma^5 \\
    1 & -\gamma & \gamma^2 & -\gamma^3 & \gamma^4 & -\gamma^5 \\
    0 & g_{32} & g_{33} & g_{34} & g_{35} & g_{36} \\
    0 & g_{42} & g_{43} & g_{44} & g_{45} & g_{46} \\
    2 & g_{52} & -\gamma^2-3 & g_{54} & -\gamma^4+21 & g_{56} \\
    2 & g_{62} & -\gamma^2-3 & g_{64} & -\gamma^4+21 & g_{66}
    \end{pmatrix}
\end{equation}

$$g_{32}=-\gamma^5-6\gamma^3-12\gamma,\quad g_{33}=18\gamma^4+69\gamma^2-57,\quad g_{34}=-12{\gamma}^{5}-33{\gamma}^{3}+90\gamma,
$$ $$g_{35}=-69{\gamma}^{4}-258{\gamma}^{2}+225,\quad 
g_{36}=45{\gamma}^{5}+114{\gamma}^{3}-396\gamma,$$ $$g_{42}=\gamma^5+6\gamma^3+12\gamma,\quad g_{43}=18\gamma^4+69\gamma^2-57,\quad g_{44}=12{\gamma}^{5}+33{\gamma}^{3}-90\gamma,$$ $$g_{45}=-69{\gamma}^{4}-258{\gamma}^{2}+225,\quad 
g_{46}=-45{\gamma}^{5}-114{\gamma}^{3}+396\gamma,$$ $$g_{52}=-\gamma^5-4\gamma^3+2\gamma,\quad g_{54}=6\gamma^5+23\gamma^3-18\gamma,\quad g_{56}=-25\gamma^5-96\gamma^3+72\gamma,$$ $$g_{62}=\gamma^5+4\gamma^3-2\gamma,\quad g_{64}=-6\gamma^5-23\gamma^3+18\gamma,\quad g_{66}=25\gamma^5+96\gamma^3-72\gamma.$$

\subsubsection*{$f(x)=x^3+6x+3$}

\begin{equation}\label{fifthgrammatrix}
    G(H_W,L_B)=\begin{pmatrix}
    1 & \gamma & \gamma^2 & \gamma^3 & \gamma^4 & \gamma^5 \\
    1 & -\gamma & \gamma^2 & -\gamma^3 & \gamma^4 & -\gamma^5 \\
    0 & g_{32} & g_{33} & g_{34} & g_{35} & g_{36} \\
    0 & g_{42} & g_{43} & g_{44} & g_{45} & g_{46} \\
    2 & g_{52} & -\gamma^2+12 & g_{54} & -\gamma^4+168 & g_{56} \\
    2 & g_{62} & -\gamma^2+12 & g_{64} & -\gamma^4+168 & g_{66}
    \end{pmatrix}
\end{equation}

$$g_{32}=-4{\gamma}^{5}+54{\gamma}^{3}-33\gamma,\quad g_{33}=-120{\gamma}^{4}+1497{\gamma}^{2}+732,\quad g_{34}=66{\gamma}^{5}-768{\gamma}^{3}-1116\gamma,
$$ $$g_{35}=-1497{\gamma}^{4}+18672{\gamma}^{2}+9144,\quad 
g_{36}=801{\gamma}^{5}-9276{\gamma}^{3}-14148\gamma,$$ $$g_{42}=4{\gamma}^{5}-54{\gamma}^{3}+33\gamma,\quad g_{43}=-120{\gamma}^{4}+1497{\gamma}^{2}+732,\quad g_{44}=-66{\gamma}^{5}+768{\gamma}^{3}+1116\gamma,$$ $$g_{45}=-1497{\gamma}^{4}+18672{\gamma}^{2}+9144,\quad 
g_{46}=-801{\gamma}^{5}+9276{\gamma}^{3}+14148\gamma,$$ $$g_{52}=-4{\gamma}^{5}+50{\gamma}^{3}+23\gamma,\quad g_{54}=-54{\gamma}^{5}+674{\gamma}^{3}+324\gamma,$$ $$ g_{56}=-697{\gamma}^{5}+8700{\gamma}^{3}+4176
\gamma,\quad g_{62}=4{\gamma}^{5}-50{\gamma}^{3}-23\gamma,$$ $$ g_{64}=54{\gamma}^{5}-674{\gamma}^{3}-324\gamma,\quad g_{66}=697{\gamma}^{5}-8700{\gamma}^{3}-4176.$$

\subsection*{Product of generators}

\subsubsection*{$f(x)=x^5+15x^2+5$}

\begin{equation}\label{coordproductfirstpolyn5}
    \beta'=\left(\begin {array}{c} {\epsilon_1}\,{\delta_1}+{\frac {1716415065\,{\epsilon_2}\,{\delta_1}}{642386524}}+{\frac {344708445\,{\epsilon_3}\,{\delta_1}}{24707174}}+{\frac {13637321545\,{\epsilon_4}\,{\delta_1}}{642386524}}+150\,{\epsilon_5}\,{\delta_1}\\ \noalign{\medskip}{\frac {114902815\,{\epsilon_1}\,{\delta_2}}{98828696\,t}}+{\frac {13637321545\,{\epsilon_2}\,{\delta_2}}{7708638288\,t}}+{\frac{25}{2}}\,{\frac {{\epsilon_3}\,{\delta_2}}{t}}+{\frac {3045443610\,{\epsilon_4}\,{\delta_2}}{160596631\,t}}+{\frac {33002652875\,{\epsilon_5}\,{\delta_2}}{296486088\,t}}\\ \noalign{\medskip}-{\frac {1236455183\,{\epsilon_2}\,{\delta_1}}{5139092192}}-{\frac {3740722775\,{\epsilon_3}\,{\delta_1}}{3557833056}}-{\frac {12064716635\,{\epsilon_4}\,{\delta_1}}{30834553152}}-{\frac {805\,{\epsilon_5}\,{\delta_1}}{48}}\\ \noalign{\medskip}-{\frac {3740722775\,{\epsilon_1}\,{\delta_2}}{42693996672\,t}}-{\frac {12064716635\,{\epsilon_2}\,{\delta_2}}{370014637824\,t}}-{\frac {805\,{\epsilon_3}\,{\delta_2}}{576\,t}}+{\frac {153714326975\,{\epsilon_4}\,{\delta_2}}{1110043913472\,t}}-{\frac {25112657245\,{\epsilon_5}\,{\delta_2}}{1778916528\,t}}\\ \noalign{\medskip}{\frac {429885670895\,{\epsilon_2}\,{\delta_1}}{6660263480832}}+{\frac {15317808053\,{\epsilon_3}\,{\delta_1}}{170775986688}}+{\frac {1690569125905\,{\epsilon_4}\,{\delta_1}}{4440175653888}}+{\frac {18155\,{\epsilon_5}\,{\delta_1}}{10368}}\\ \noalign{\medskip}{\frac {15317808053\,{\epsilon_1}\,{\delta_2}}{2049311840256\,t}}+{\frac {1690569125905\,{\epsilon_2}\,{\delta_2}}{53282107846656\,t}}+{\frac {18155\,{\epsilon_3}\,{\delta_2}}{124416\,t}}+{\frac {2615627812535\,{\epsilon_4}\,{\delta_2}}{17760702615552\,t}}+{\frac {457036180625\,{\epsilon_5}\,{\delta_2}}{256163980032\,t}}\\ \noalign{\medskip}-{\frac {861830682163\,{\epsilon_2}\,{\delta_1}}{213128431386624}}-{\frac {5731464085\,{\epsilon_3}\,{\delta_1}}{1366207893504}}-{\frac {1203824325611\,{\epsilon_4}\,{\delta_1}}{35521405231104}}-{\frac {30685\,{\epsilon_5}\,{\delta_1}}{331776}}\\ \noalign{\medskip}-{\frac {5731464085\,{\epsilon_1}\,{\delta_2}}{16394494722048\,t}}-{\frac {1203824325611\,{\epsilon_2}\,{\delta_2}}{426256862773248\,t}}-{\frac {30685\,{\epsilon_3}\,{\delta_2}}{3981312\,t}}-{\frac {5849774679485\,{\epsilon_4}\,{\delta_2}}{284171241848832\,t}}-{\frac {3525523930915\,{\epsilon_5}\,{\delta_2}}{32788989444096\,t}}\\ \noalign{\medskip}{\frac {28089103\,{\epsilon_2}\,{\delta_1}}{213128431386624}}+{\frac {186769\,{\epsilon_3}\,{\delta_1}}{1366207893504}}+{\frac {39241535\,{\epsilon_4}\,{\delta_1}}{35521405231104}}+{\frac {{\epsilon_5}\,{\delta_1}}{331776}}\\ \noalign{\medskip}{\frac {186769\,{\epsilon_1}\,{\delta_2}}{16394494722048\,t}}+{\frac {39241535\,{\epsilon_2}\,{\delta_2}}{426256862773248\,t}}+{\frac {{\epsilon_3}\,{\delta_2}}{3981312\,t}}+{\frac {190712785\,{\epsilon_4}\,{\delta_2}}{284171241848832\,t}}+{\frac {114902815\,{\epsilon_5}\,{\delta_2}}{32788989444096\,t}}\end {array}\right)
\end{equation}

\subsubsection*{$f(x)=x^5+10x^2+5$}

\begin{equation}\label{coordproductsecondpolyn5}
    \beta'=\left(\begin {array}{c} {\epsilon_1}\,{\delta_1}-{\frac {4066158490\,{\epsilon_2}\,{\delta_1}}{954111429}}+{\frac {11621900\,{\epsilon_3}\,{\delta_1}}{2321439}}+{\frac {32953801295\,{\epsilon_4}\,{\delta_1}}{954111429}}-200\,{\epsilon_5}\,{\delta_1}\\ \noalign{\medskip}{\frac {2905475\,{\epsilon_1}\,{\delta_2}}{4642878\,t}}+{\frac {32953801295\,{\epsilon_2}\,{\delta_2}}{7632891432\,t}}-25\,{\frac {{\epsilon_3}\,{\delta_2}}{t}}+{\frac {87851681855\,{\epsilon_4}\,{\delta_2}}{1908222858\,t}}+{\frac {2423056375\,{\epsilon_5}\,{\delta_2}}{18571512\,t}}\\ \noalign{\medskip}-{\frac {2565694281\,{\epsilon_2}\,{\delta_1}}{3392396192}}+{\frac {949487275\,{\epsilon_3}\,{\delta_1}}{297144192}}-{\frac {61041394225\,{\epsilon_4}\,{\delta_1}}{20354377152}}-{\frac {445\,{\epsilon_5}\,{\delta_1}}{16}}\\ \noalign{\medskip}{\frac {949487275\,{\epsilon_1}\,{\delta_2}}{2377153536\,t}}-{\frac {61041394225\,{\epsilon_2}\,{\delta_2}}{162835017216\,t}}-{\frac {445\,{\epsilon_3}\,{\delta_2}}{128\,t}}+{\frac {6092208052225\,{\epsilon_4}\,{\delta_2}}{325670034432\,t}}-{\frac {18542187595\,{\epsilon_5}\,{\delta_2}}{594288384\,t}}\\ \noalign{\medskip}{\frac {760219986455\,{\epsilon_2}\,{\delta_1}}{15632161652736}}-{\frac {546579773\,{\epsilon_3}\,{\delta_1}}{3169538048}}+{\frac {223522035595\,{\epsilon_4}\,{\delta_1}}{1954020206592}}+{\frac {28825\,{\epsilon_5}\,{\delta_1}}{16384}}\\ \noalign{\medskip}-{\frac {546579773\,{\epsilon_1}\,{\delta_2}}{25356304384\,t}}+{\frac {223522035595\,{\epsilon_2}\,{\delta_2}}{15632161652736\,t}}+{\frac {28825\,{\epsilon_3}\,{\delta_2}}{131072\,t}}-{\frac {66180799869005\,{\epsilon_4}\,{\delta_2}}{62528646610944\,t}}+{\frac {36935998025\,{\epsilon_5}\,{\delta_2}}{25356304384\,t}}\\ \noalign{\medskip}-{\frac {21048029533\,{\epsilon_2}\,{\delta_1}}{13895254802432}}+{\frac {338220965\,{\epsilon_3}\,{\delta_1}}{76068913152}}-{\frac {11879343659\,{\epsilon_4}\,{\delta_1}}{5210720550912}}-{\frac {7115\,{\epsilon_5}\,{\delta_1}}{131072}}\\ \noalign{\medskip}{\frac {338220965\,{\epsilon_1}\,{\delta_2}}{608551305216\,t}}-{\frac {11879343659\,{\epsilon_2}\,{\delta_2}}{41685764407296\,t}}-{\frac {7115\,{\epsilon_3}\,{\delta_2}}{1048576\,t}}+{\frac {4822141269325\,{\epsilon_4}\,{\delta_2}}{166743057629184\,t}}-{\frac {20676257105\,{\epsilon_5}\,{\delta_2}}{608551305216\,t}}\\ \noalign{\medskip}{\frac {26624495\,{\epsilon_2}\,{\delta_1}}{1000458345775104}}-{\frac {47531\,{\epsilon_3}\,{\delta_1}}{608551305216}}+{\frac {5007295\,{\epsilon_4}\,{\delta_1}}{125057293221888}}+{\frac {{\epsilon_5}\,{\delta_1}}{1048576}}\\ \noalign{\medskip}-{\frac {47531\,{\epsilon_1}\,{\delta_2}}{4868410441728\,t}}+{\frac {5007295\,{\epsilon_2}\,{\delta_2}}{1000458345775104\,t}}+{\frac {{\epsilon_3}\,{\delta_2}}{8388608\,t}}-{\frac {2033079245\,{\epsilon_4}\,{\delta_2}}{4001833383100416\,t}}+{\frac {2905475\,{\epsilon_5}\,{\delta_2}}{4868410441728\,t}}\end {array}\right)
\end{equation}

\end{document}